\documentclass[12pt,a4paper]
{article}

\usepackage[ngerman,english]{babel}
\usepackage[utf8]{inputenc}

\usepackage{amsmath,amssymb,amsthm}

\usepackage{enumitem}
\usepackage{dsfont}
\usepackage{color}

\usepackage[hyperfootnotes=false]{hyperref} 
\usepackage[a4paper,left=20mm,right=15mm, top=25mm, bottom=25mm]{geometry}

\newtheorem{theorem}{Theorem}[section]
\newtheorem{lemma}[theorem]{Lemma}

\theoremstyle{definition}

\newcommand{\ind}{\mathds{1}} 
\newcommand{\C}{\mathbb{C}} 
\newcommand{\E}{\mathbb{E}} 
\newcommand{\R}{\mathbb{R}} 
\newcommand{\Z}{\mathbb{Z}} 
\newcommand{\N}{\mathbb{N}} 
\newcommand{\on}{\operatorname} 
\newcommand{\LK}{\mathcal{L}} 
\newcommand{\A}{\mathcal{A}}

\newcommand{\Co}{\on{Cov}}

\begin{document}
\author{Dennis Müller}
\title{A Central Limit Theorem for Lipschitz-Killing Curvatures of Gaussian Excursions}
\maketitle
\begin{abstract}
This paper studies the excursion set of a real stationary isotropic Gaussian random field above a fixed level. We show that the standardized Lipschitz-Killing curvatures of the intersection of the excursion set with a window converges in distribution to a normal distribution as the window grows to the $d$-dimensional Euclidean space. Moreover a lower bound for the asymptotic variance is derived.
\end{abstract}

\let\thefootnote\relax\footnote{
\textit{2010 Mathematics Subject Classification:} Primary: 60F05; Secondary: 60D05; 60G60; 60G15.

\textit{Keywords:} Gaussian random field, excursion set, Lipschitz-Killing curvature, central limit theorem, Malliavin-Stein method, Wiener Chaos expansion.}

\section{Introduction}

Let $X=\{X(t)\mid t\in \R^d\}$ be a real Gaussian random field defined on a probability space $(\Omega, \mathcal{F},\mathbb{P})$. The excursion set of $X$ for the level $u\in\R$ is the random set
\begin{align*}
  X^{-1}([u,\infty))=\{t\in \R^d\mid X(t)\geq u\},
\end{align*}
whose properties are an active area of research, cf. \cite{book:AdlerTaylor}, \cite{book:AzaisWschebor}, \cite{paper:LachiezeRey}, \cite{paper:AdlerMoldavskayaSamorodnitsky} among others. As a stochastic model, random fields have many applications, for instance in human brain mapping (\cite{paper:CaoWorsley}), astrophysics (\cite{book:LiddleLyth}) and optics (\cite{paper:BerryDennis}).

To gain a deeper understanding of random excursion sets, several geometric characteristics can be used. In this paper we generalize results for the Euler-Poincar\'e characteristic to the so-called Lipschitz-Killing curvatures $\LK_m$, which are given for $m=0,\ldots ,d-1$ and $M\subset \R^d$ closed with nonempty interior, $\mathcal{C}^2$ boundary and induced Riemannian structure by
\begin{align*}
  \LK_{m}(M)= \frac{1}{\omega_{d-m}}\int_{\partial M} \on{detr}_{d-1-m}(S_{E_d}(E_i,E_j))_{i,j=1}^{d-1} \,d\mathcal{H}^{d-1},
\end{align*}
where $(E_i)_{i=1\,\ldots,d-1}$ denotes an orthonormal frame field on $\partial M$, $E_d$ denotes the inward normal, $S$ denotes the scalar second fundamental form, $\on{detr}_{d-1-m}(A)$ denotes the sum over all $(d-1-m)\times (d-1-m)$ principal minors of $A$, the constant $\omega_{d-m}$ denotes the surface area of the $(d-m-1)$-dimensional unit sphere $S^{d-m-1}$ and  $\mathcal{H}^{d-1}$ denotes the $(d-1)$\nobreakdash-\hspace{0pt}dimensional Hausdorff measure. For further details see \cite [(10.7.6)]{book:AdlerTaylor}, and  in particular, \cite [Section 10.7]{book:AdlerTaylor} for the more complex framework of Whitney stratified spaces considered in this paper. For special choices of $m$ the Lipschitz-Killing curvatures describe simple geometric features of the set like the volume ($m=d$), half the surface area ($m=d-1$) and the Euler-Poincar\'e characteristic ($m=0$).

The aim of this work is to establish a central limit theorem for the standardized $m$\nobreakdash-\hspace{0pt}th Lipschitz-Killing curvature of the intersection of an excursion set for the level $u$ of a stationary isotropic Gaussian random field with an open ball $B^d_N$ of radius $N$, as $N$ goes to infinity, that is
\begin{align*}
  \frac{\LK_{m}\left (B^d_N\cap X^{-1}([u,\infty))\right)-\E\left[\LK_{m}\left (B^d_N\cap X^{-1}([u,\infty))\right) \right]}{(\LK_d(B^d_N))^{\frac{1}{2}}} \underset{N\to \infty}{\overset{\mathcal{D}}{\longrightarrow}} \mathcal{N}(0,\sigma^2_m)
\end{align*}
for some $\sigma^2_m\geq0$, where a lower bound for $\sigma_m^2$ is given in Lemma \ref{lem:lowerBound}. The present paper generalizes the work of \cite{paper:EstradeLeon}, where such a CLT is established for $m=0$. The case $m=d-1$ and $d=2$ is treated in \cite{paper:KratzLeon}. For the case $m=d$ of the volume, the central limit theorem holds under weaker requirements than Gaussianity, for instance, for quasi-associated random fields, PA- or NA-random fields, Max- or $\alpha$-stable fields, cf. the survey \cite{review:Spodarev} and the references therein. For this reason we concentrate on the cases $m=0,\ldots,d-1$ in this work.

We pursue the following strategy of proof. First, we apply the Crofton formula from integral geometry to express the $m$-th Lipschitz-Killing curvature of a sufficiently regular set $M\subset \R^d$ as an integral average of the Euler-Poincar\'e characteristics of the intersections of $M$ with affine $(d-m)$-flats, where the integration is with respect to the suitably normalized motion invariant measure $\mu$ over the affine Grassmannian $A(d,d-m)$ of all affine $(d-m)$-flats of $\R^d$ (cf. \cite[Thm. 13.1.1]{book:AdlerTaylor}).
An application to $M=B^d_N\cap X^{-1}([u,\infty))$ leads to the investigation of the Euler-Poincar\'e characteristic of the intersection of the Gaussian excursion with a lower dimensional ball in an affine subspace. By Morse Theory (cf. \cite [Corollary 9.3.5]{book:AdlerTaylor}), this characteristic can be expressed as a difference of counting variables.
Inspired by the ideas of \cite{paper:EstradeLeon}, we use a refinement of the approach in \cite{paper:EstradeLeon} to control the dependence of the counting variables on the affine flat. That is, we use Rice's formulas, cf. \cite [Chapter 6]{book:AzaisWschebor}, \cite [Section 11.2]{book:AdlerTaylor}, in the affine flat to obtain a Hermite expansion of the $m$-th Lipschitz-Killing curvature via an approximation argument. This Hermite expansion leads to a representation of $\LK_{m}\left (B^d_N\cap X^{-1}([u,\infty))\right)$ by stochastic integrals, to which we apply results from the theory of normal approximation based on Stein's method and Malliavin calculus as described in \cite{book:NourdinPeccati}.

The basic tool of our approach, the Wiener chaos expansion, was already prominent in the works of \cite{paper:ChambersSlud}, \cite{paper:Slud} and \cite{paper:KratzLeon1997}, to mention just a few. This access to normal approximations is very popular and is used in various settings similar to ours, cf. \cite{paper:AdlerNaitzat}, who show a central limit theorem for the Euler integration of random functions, \cite{paper:CammarotaMarinucci}, who investigate Gaussian excursions on the $2$-sphere or  \cite{paper:Nicolaescu}, who studies critical points of random Fourier series on the $m$-dimensional torus.

Although less explicit, the results of this paper might be compared with recent progress in the second order analysis of the Boolean model, another fundamental model of stochastic geometry, cf. \cite{paper:HugLastSchulte}, \cite{book:LastPenrose}. This progress is largely based on the Malliavin calculus for general Poisson processes.

\section{Main Theorem}
We impose the following conditions on a given real random field $X=\{X(t)\mid t\in \R^d\}$.
\begin{enumerate}[label= (A\arabic*)]
  \item \label{as:dif} $X$ is a centered, stationary, isotropic Gaussian field. The trajectories are almost surely of class $\mathcal{C}^3$. The abbreviation $\Co^X(t):=\E\left[ X(t)X(0)\right]$, $t\in \R^d$, denotes the covariance function of $X$, which satisfies $\on{Cov}^X(0)=1$ and $-D^2 \on{Cov}^X(0)=I_d$.

  \item \label{as:nondegenerate} For $0\neq t\in \R^d$ the covariance matrix of the vector
    \begin{align*}
      \left(X(t),\left(\frac{\partial}{\partial t_i} X(t)\right)_{i=1}^d,\left(\frac{\partial^2}{\partial t_i \partial t_j }X(t)\right)_{1\leq i\leq j\leq d},\left(\frac{\partial}{\partial t_i} X(0)\right)_{i=1}^d\right)
    \end{align*}
    has full rank.
  \item\label{as:covariance} The mapping defined by
    \begin{align*}
       \psi(t):=\max\left \{\left |\frac{\partial^{k}}{\partial t_{j_1}\ldots \partial t_{j_k}}\on{Cov}^X(t)\right |:k\in\{0,\ldots, 4\},1\leq j_1,\ldots,j_k \leq d\right\}
     \end{align*}
     for $t\in \R^d$, satisfies
     \begin{align*}
       \psi(t)\stackrel{\|t\|\to \infty}{\longrightarrow}0\text{ and } \psi \in L^1(\R^d).
     \end{align*}
\end{enumerate}
We heavily rely on \ref{as:dif} in several places, for instance in the proof of Lemma \ref{lem:decom} and in the calculations in the appendix for Lemma \ref{lem:hilfL^2}. If \ref{as:nondegenerate} holds, the conditions on the covariance from \ref{as:dif} are always satisfied after normalizing the Gaussian field. We believe that it is enough to assume $\mathcal{C}^2$ regularity and an integrability condition on $\Co^X$, cf. \cite{paper:EstradeFournier}, but stick to the $\mathcal{C}^3$ assumption to smoothen the computations of the appendix. Under the differentiability assumptions of \ref{as:dif}, the condition \ref{as:nondegenerate} ensures that the paths of $X$ are almost surely Morse functions and allows us to perform calculations involving Gaussian regressions. Condition \ref{as:covariance} implies that the conditions for a central limit theorem are satisfied.
Note that from \ref{as:covariance} we obtain that $\psi \in L^q(\R^d)$, $q\in \N$, and moreover that $X$ admits a continuous spectral density, cf. \cite [Theorem 2.§12.3 (Inversion Formula)]{book:Shiryaev}. Furthermore the mapping defined by
\begin{align*}
  \widetilde\psi(t):=\sup\left \{\left |\frac{\partial^{k}}{\partial v_1\ldots \partial v_k}\on{Cov}^X(t)\right |:k\in\{0,\ldots, 4\},v_1,\ldots,v_k\in S^{d-1}\right\}, \quad t\in \R^d
\end{align*}
satisfies $\widetilde\psi(t)\leq d^2\psi(t)$, for $t\in \R^d$, and therefore is also in $L^q(\R^d)$, $q\in \N$.

Let $u\in \R$ be the level of the considered excursion set and denote by $B^d_N\subset \R^d$ the open ball with radius $N\in \N$ centered at the origin.
We proof the following central limit theorem.
\begin{theorem}\label{thm:mainTheorem}
  Let $X$ be a real Gaussian field on $\R^d$, which satisfies the assumptions (A1)--(A3) and let $m\in\{0,\ldots,d-1\}$. Then the $m$-th Lipschitz-Killing curvature $\LK_m$ of the excursion set for the level $u\in \R$ satisfies
  \begin{align*}
    \frac{\LK_{m}\left (B^d_N\cap X^{-1}([u,\infty))\right)-\E\left[\LK_{m}\left (B^d_N\cap X^{-1}([u,\infty))\right) \right]}{\mathcal{H}^d(B^d_N)^{\frac{1}{2}}} \stackrel{\mathcal{D}}{\longrightarrow} \mathcal{N}(0,\sigma^2_m)
  \end{align*}
  for $N\to \infty$ and some $\sigma^2_m\geq0$.
\end{theorem}

A lower bound for the asymptotic variance $\sigma_m^2$ is shown in Lemma \ref{lem:lowerBound}.


\section{Approximation of Lipschitz-Killing curvatures}

We fix the following notation. Let $f\colon\R^d \to \R$ be a mapping of class $\mathcal{C}^2$. We denote for $t\in \R^d$ by $\nabla f(t)$ and $D^2 f(t)$ the gradient and the $d\times d$-matrix $(\frac{\partial ^2}{\partial t_i \partial t_j} f)_{1\leq i,j \leq d}$ of second derivatives  of $f$, respectively. For $F\in A(d,d-m)$ we denote by $F^\circ$ the directional space of $F$, which is an element in the Grassmannian $G(d,d-m)$ of $(d-m)$-dimensional linear subspaces of $\R^d$. The motion invariant measure $\nu$ on $G(d,d-m)$ is normalized such that $\nu(G(d,d-m))=\binom{d}{d-m}\frac{\omega_{d}}{\omega_{m}\omega_{d-m}}$. We denote by $v(F):=(v_1,\ldots,v_{d-m})$ an orthonormal basis of $F^\circ$ and define the gradient of $f$ in $F$ as the vector field given by
\begin{align*}
  \nabla(f|_F)(t):=\sum_{i=1}^{d-m} \frac{\partial}{\partial v_i} f(t) v_i,
\end{align*}
for $t\in F$, where $\frac{\partial}{\partial v_i}$ denotes the directional derivative in direction $v_i$. The second derivative of $f$ in the affine flat $F$ and in point $t\in F$ is defined as the linear mapping on $F^\circ$ given by the $d\times d$-matrix
\begin{align*}
  D^2(f|_F)(t) := \left (\begin{array}{@{}c|c|c@{}}
      v_1 & \cdots & v_{d-m} \\
      \end{array}\right ) \left(\frac{\partial^2}{\partial v_i\partial v_j}f(t)\right)_{1\leq i,j\leq d-m}\left (\begin{array}{@{}c|c|c@{}}
      v_1 & \cdots & v_{d-m} \\
      \end{array}\right )^\top.
\end{align*}
We note that these definitions coincide with the Riemannian ones, using for $F$ the coordinate map $\varphi\colon F\to \R^{d-m}$ given by $v\mapsto (v_1| \ldots |v_{d-m})^\top v$ and therefore do not depend on the choice of $v(F)$. Moreover, we define
\begin{align} \label{def:nablaInF}
  \nabla_{v(F)} f \colon \R^d \to \R^{d-m},\quad t \mapsto \left ( \frac{\partial}{\partial v_i} f(t)\right)_{i=1}^{d-m},
\end{align}
whose components are the coefficients of $\nabla(f|_F)$ in the basis $v(F)$, as well as
\begin{align}\label{def:D2InF}
  D^2_{v(F)} f \colon \R^d \to \R^{(d-m)\times (d-m)},\quad t\mapsto \left(\frac{\partial^2}{\partial v_i \partial v_j}f(t)\right)_{i,j=1}^{d-m}.
\end{align}

Using standard results from \cite{book:AdlerTaylor}, we now derive a more practical representation of the $m$-th Lipschitz-Killing curvature $\LK_m$ of the excursion set in $B^d_N$.

We define $\kappa_m:=\mathcal{H}^m(B^m_1)$, $m\in \N$, and consider for $\varepsilon >0$ the mapping
\begin{align*}
  \delta_\varepsilon \colon \R^{d}\to\R, \quad  x\mapsto \frac{1}{\varepsilon^{d-m}\kappa_{d-m}}\ind_{B^{d}_\varepsilon}(x),
\end{align*}
which is a Dirac sequence for $\varepsilon \to 0$ on every $(d-m)$-dimensional linear subspace $E$ of $\R^d$, that is, for each continuous mapping $f\colon E\to \R$, we have
\begin{align*}
  \lim_{\varepsilon \to 0} \int_{E} \delta_\varepsilon(x) f(x) \,dx= f(0).
\end{align*}

We apply the Crofton formula in \cite [Thm. 13.1.1]{book:AdlerTaylor}, to obtain
\begin{align*}
  \LK_m\left (B^d_N\cap X^{-1}([u,\infty))\right)&=\LK_m\left (\overline{B^d_N}\cap X^{-1}([u,\infty))\right )-\LK_m\left (S^{d-1}_N\cap X^{-1}([u,\infty))\right)\\
  &=\int_{A(d,d-m)}\LK_0\left (\overline{B^d_N}\cap X^{-1}([u,\infty))\cap F\right ) \,\mu(dF)\\
  &\quad-\int_{A(d,d-m)}\LK_0\left (S^{d-1}_N\cap X^{-1}([u,\infty))\cap F\right) \,\mu(dF)\\
  &=\int_{A(d,d-m)}\LK_0\left (B^d_N\cap X^{-1}([u,\infty))\cap F\right) \,\mu(dF).
\end{align*}
By the assumptions made, we know that the trajectories of $X$ are almost surely Morse functions on $B^d_N\cap F$, for $\mu$ almost all $F$, cf. \cite [Definition 9.3.1]{book:AdlerTaylor} and Lemma \ref{lem:pre}. Therefore, restricting the integration to a suitable subset $A'\subset A(d,d-m)$ as provided by Lemma \ref{lem:pre}, we can apply \cite [Cor. 9.3.5]{book:AdlerTaylor} to the above integrand, to see that
\begin{align*}
  \LK_0&(B^d_N\cap X^{-1}([u,\infty))\cap F) \\
  &= \#\{t\in B^d_N\cap F:X(t)\geq u, \nabla (X|_F)(t)=0,\iota_{-X,B^d_N\cap F}(t) \text{ even}\}\\
  &\quad- \#\{t\in B^d_N\cap F:X(t)\geq u, \nabla (X|_F)(t)=0,\iota_{-X,B^d_N\cap F}(t) \text{ odd}\},
\end{align*}
where $\iota$ denotes the tangential Morse index, cf. \cite [(9.1.2)]{book:AdlerTaylor}. Later computations will benefit from a more general definition in which we define the latter random variable for a bounded, convex window $W\subset \R^d$, and thus define
\begin{align}\label{def:zeta}
  \zeta_{m,W}&:=\int_{A(d,d-m)}\#\{t\in W\cap F:X(t)\geq u, \nabla (X|_F)(t)=0,\iota_{-X,W\cap F}(t) \text{ even}\}\notag\\
  &\quad- \#\{t\in W\cap F:X(t)\geq u, \nabla (X|_F)(t)=0,\iota_{-X,W\cap F}(t) \text{ odd}\}\,\mu(dF).
\end{align}
Motivated by the use of a Dirac sequence to approximate these counting variables, cf. \cite [Lemma 11.2.10]{book:AdlerTaylor}, we introduce the approximation
\begin{align}\label{def:zetaEpsilon}
  \zeta^\varepsilon_{m,W} := (-1)^{d-m}\int_{A(d,d-m)} \int_{W\cap F} \delta_\varepsilon(\nabla (X|_F)(t))\ind\{X(t)\geq u\}\det(D^2(X|_F)(t)) \,dt\, \mu(dF)
\end{align}
and now specify the quality of this approximation. We first need the following Lemma, whose proof is postponed to the appendix:

\begin{lemma}\label{lem:hilfL^2}
  Let $D\subset\R^d$ be compact, assume \ref{as:dif} and \ref{as:nondegenerate} and let $W\subset \R^d$ be convex and bounded. Then the following is true:
   \begin{enumerate}[label={\normalfont(\roman*)}]
     \item There is a constant $c>0$, depending on $X$, $d$, $m$, and $W$, such that for $F\in A(d,d-m)$ and $y\in F^\circ\cap D $
           \begin{align*}
              \E\left[\# \{t\in W \cap F \colon \nabla (X|_F)(t)=y\}^2 \right]< c.
           \end{align*}
     \item For all $F\in A(d,d-m)$ the mapping
           \begin{align*}
             y\mapsto \E\left[ \# \{t\in W \cap F \colon \nabla (X|_F)(t)=y\}^2\right]
           \end{align*}
           is continuous on $F^\circ \cap D$.
     \item For all $F\in A(d,d-m)$, we have
           \begin{align*}
             \xi_W(F&,\varepsilon)\overset{L^2(\mathbb{P})}{\underset{\varepsilon \to 0}{\longrightarrow}} \xi_W(F),
           \end{align*}
           where
           \begin{align*}
             \xi_W(F,\varepsilon)&:=(-1)^{d-m}\int_{W\cap F} \delta_\varepsilon(\nabla (X|_F)(t))\ind\{X(t)\geq u\}\det(D^2(X|_F)(t))\, dt,\\
             \xi_W(F)&:=\#\{t\in W\cap F:X(t)\geq u, \nabla (X|_F)(t)=0,\iota_{-X,W\cap F}(t) \mathrm{\, even}\}\\
             &\quad- \#\{t\in W\cap F:X(t)\geq u, \nabla (X|_F)(t)=0,\iota_{-X,W\cap F}(t) \mathrm{\, odd}\}.
           \end{align*}
   \end{enumerate}
\end{lemma}
We now show that the approximation $\zeta_{m,W}^\varepsilon$ is indeed an approximation of the variable of interest $\zeta_{m,W}$.
\begin{lemma}\label{lem:L^2approx}
  Let $(X_t)_{t\in\R^d}$ be a real-valued Gaussian field satisfying \ref{as:dif} and \ref{as:nondegenerate} and let $W\subset\R^d$ be convex and bounded. Then
  \begin{align*}
    \zeta^\varepsilon_{m,W} \stackrel{L^2(\mathbb{P})}{\longrightarrow} \zeta_{m,W}
  \end{align*}
  for $\varepsilon \to 0$, where $\zeta_{m,W}$ and $\zeta_{m,W}^\varepsilon$ are defined by \eqref{def:zeta} and \eqref{def:zetaEpsilon}, respectively.
\end{lemma}

\begin{proof}
  By Jensen's inequality and Fubini's theorem
  \begin{align*}
    \E\left[ \left (\zeta_{m,W}-\zeta^\varepsilon_{m,W}\right)^2\right] &\leq c \E\left[ \int_{A(d,d-m)}\left ( \xi_W(F)-\xi_W(F,\varepsilon)\right)^2\,\mu(dF)\right]\\
    &=c \int_{A(d,d-m)}\E\left[\left ( \xi_W(F)-\xi_W(F,\varepsilon)\right)^2\right]\, \mu(dF),
  \end{align*}
  where $c=\mu(\{F:F\cap W\neq \emptyset\})\leq {d \brack d-m } \on{diam}(W)^m\kappa_m$, cf. \cite [(6.3.12)]{book:AdlerTaylor} for the definition of the flag coefficients.
  Thus, if we justify changing the order of the limits $\lim_{\varepsilon \to 0}$ and $\int_{A(d,d-m)}$, we are done by Lemma \ref{lem:hilfL^2} (iii). In order to apply the dominated convergence theorem, we bound the integrand by an integrable function, not depending on $\varepsilon$. Observe that
  \begin{align*}
    \E\left[\left ( \xi_W(F)-\xi_W(F,\varepsilon)\right)^2\right]&\leq 2 \E\left[ \#\{t\in W\cap F : \nabla(X|_F)(t)=0\}^2\right] \\
    &\quad +2 \E\left[ \left(\int_{W\cap F}\delta_\varepsilon(\nabla (X|_F)(t))|\det(D^2(X|_F)(t))| \,dt\right)^2\right].
  \end{align*}
  For the first term Lemma \ref{lem:hilfL^2} (i) yields
  \begin{align*}
    \E\left[ \#\{t\in W \cap F : \nabla(X|_F)(t)=0\}^2\right] \leq c \ind\{F\cap W \neq \emptyset \},
  \end{align*}
  where $c>0$ is a constant depending on $X$,$d$, $m$ and $W$. For the second term, we apply the coarea formula to $\nabla(X|_F)$, cf. \cite [Theorem 3.2.12]{book:Federer}, then Jensen's inequality to the measure $\ind\{y\in F^\circ\}\delta_\varepsilon(y) \mathcal{H}^{d-m}(dy)$ followed by Fubini's theorem, to obtain
  \begin{align*}
    \E&\left[ \left(\int_{W\cap F}\delta_\varepsilon(\nabla (X|_F)(t))|\det(D^2(X|_F)(t))| \,dt\right)^2\right] \\
    &\leq \int_{F^\circ} \E\left[\#\{t\in W\cap F : \nabla(X|_F)(t)=y\}^2\right]\delta_\varepsilon(y)\,dy.
  \end{align*}
  Again by Lemma \ref{lem:hilfL^2} (i), we can bound this for all $\varepsilon \leq1$ by the expression
  \begin{align*}
    c \int_{F^\circ} \delta_{\varepsilon}(y) \,dy \ind\{F\cap W \neq \emptyset \} = c\ind\{F\cap W \neq \emptyset \}.
  \end{align*}
  Both bounds are independent of $\varepsilon$ and integrable with respect to $\mu$, which shows the assertion. \qedhere
\end{proof}

Before we move on with the main proof, we show the following lemma to obtain
a more concrete representation of $\zeta_{m,W}^\varepsilon$. We note that the special choice of the orthonormal basis $v(F)$ of $F^\circ$, for $F\in A(d,d-m)$, is irrelevant.
\begin{lemma} \label{lem:repcoord}
  Let $\varepsilon >0$, $W\subset \R^d$ be convex and bounded and assume \ref{as:dif}. Then
  \begin{align*}
  \zeta_{m,W}^\varepsilon= (-1)^{d-m}\int_{G(d,d-m)}\int_{W}&\delta_\varepsilon(\nabla_{v(F)}X(t)) \ind\{X(t)\geq u\} \det\left(D^2_{v(F)}X(t)\right)\,dt\, \nu(dF),
  \end{align*}
  where $\nabla_{v(F)}$ and $D^2_{v(F)}$ are defined in \eqref{def:nablaInF} and \eqref{def:D2InF}, respectively.
\end{lemma}

\begin{proof}
  Recall, that by definition $\nabla(f|_F)(t)=\sum_{i=1}^{d-m} \frac{\partial}{\partial v_i} f(t) v_i$ and therefore the rotation invariance of $\delta_\varepsilon$ yields
  \begin{align*}
    \delta_\varepsilon(\nabla(X|_F))= \frac{1}{\varepsilon^{d-m}\kappa_{d-m}} \ind_{B^d_\varepsilon}(\nabla(X|_F))=\delta_\varepsilon(\nabla_{v(F)}X).
  \end{align*}
  Also by definition $D^2(X|_F)(t) = \left (\begin{array}{@{}c|c|c@{}}
      v_1 & \cdots & v_{d-m} \\
      \end{array}\right ) \left(\frac{\partial^2}{\partial v_i\partial v_j}X(t)\right)_{1\leq i,j\leq d-m}\left (\begin{array}{@{}c|c|c@{}}
      v_1 & \cdots & v_{d-m} \\
      \end{array}\right )^\top$
  so that, as a linear mapping from $F^\circ$ into $F^\circ$ it has the transformation matrix $\left(\frac{\partial^2}{\partial v_i\partial v_j}X(t)\right)_{i,j=1}^{d-m}$ with respect to the chosen basis, and therefore we have
  \begin{align*}
    \det(D^2(X|_F))= \det\left(D^2_{v(F)}X\right).
  \end{align*}
  This yields with definition \eqref{def:zetaEpsilon}
  \begin{align*}
    \zeta_{m,W}^\varepsilon= (-1)^{d-m}\int_{A(d,d-m)}\int_{W\cap F}&\delta_\varepsilon(\nabla_{v(F)}X(t))\ind\{X(t)\geq u\} \det\left(D^2_{v(F)}X(t)\right)\,dt\, \mu(dF)
  \end{align*}
  and we conclude by an application of Fubini's Theorem
  \begin{align*}
    \zeta_{m,W}^\varepsilon&= (-1)^{d-m}\int_{G(d,d-m)}\int_{L^\perp}\int_{W\cap (L+y)}\delta_\varepsilon(\nabla_{v(L+y)}X(t))\ind\{X(t)\geq u\} \\
    &\hspace{4.8cm}\times \det\left(D^2_{v(L+y)}X(t)\right)\mathcal{H}^{d-m}(dt)\,\mathcal{H}^m(dy) \,\nu(dL)\\
    &=(-1)^{d-m}\int_{G(d,d-m)}\int_{L^\perp}\int_{L}\ind\{t+y\in W\}\delta_\varepsilon(\nabla_{v(L)}X(t+y))\ind\{X(t+y)\geq u\} \\
    &\hspace{4.8cm}\times \det\left(D^2_{v(L)}X(t+y)\right)\,\mathcal{H}^{d-m}(dt)\,\mathcal{H}^m(dy) \,\nu(dL)\\
    &=(-1)^{d-m}\int_{G(d,d-m)}\int_{W}\delta_\varepsilon(\nabla_{v(F)}X(t))\ind\{X(t)\geq u\} \det\left(D^2_{v(F)}X(t)\right)\,\mathcal{H}^{d}(dt) \,\nu(dF).\qedhere
  \end{align*}
\end{proof}

\section{Hermite type expansion}
 From now on, let the field $X$ satisfy the assumptions \ref{as:dif}--\ref{as:covariance}.
 We begin this section by defining for $D:=d-m+(d-m)(d-m+1)/2+1$ the $\R^{D}$-valued Gaussian random field $(\mathcal{X}^F_t\colon \Omega \to \R^D\mid(F,t)\in G(d,d-m)\times\R^d)$ by
\begin{align*}
  \mathcal{X}^F(t):=\left(\nabla_{v(F)}X(t),\left(\frac{\partial^2}{\partial v_i \partial v_j} X(t)\right)_{1\leq i \leq j \leq d-m}, X(t)\right)
\end{align*}
and denote by $\Sigma$ the covariance matrix of $\mathcal{X}^F(t)$, $(F,t)\in G(d,d-m)\times\R^d$. We note that the definition depends on the choice of $v(F)$, but considering Lemma \ref{lem:repcoord}, this does not matter.  We formulate the following lemma.
\begin{lemma}\label{lem:decom}
  The matrix $\Sigma$ is independent of $t\in \R^d$ and $F\in G(d,d-m)$. Moreover, we have $\Sigma=\Lambda\Lambda^\top$, where $\Lambda\in\on{GL}_D(\R)$ is given by $\Lambda=\begin{pmatrix}
                                                                     I_{d-m \times d-m} & 0 \\
                                                                     0 & \Lambda_2 \\
                                                                   \end{pmatrix} $,
  for some lower triangular matrix $\Lambda_2 \in GL_{D-(d-m)}(\R)$.
\end{lemma}

\begin{proof}
  By assumption \ref{as:dif} on the random field $X$, we obtain from \cite [(5.5.3), (5.7.3)]{book:AdlerTaylor} and isotropy
  \begin{align}\label{eq:covFirstDer}
    &\E\left[\frac{\partial}{\partial v_i}X(t)\frac{\partial}{\partial v_j}X(t)\right]=\E\left[\frac{\partial}{\partial t_i}X(0)\frac{\partial}{\partial t_j}X(0)\right]=\delta_{ij},\\
    &\E\left[\frac{\partial}{\partial v_i}X(t)\frac{\partial^2}{\partial v_k\partial v_l}X(t)\right]= \E\left[\frac{\partial}{\partial t_i}X(0)\frac{\partial^2}{\partial t_k\partial t_l}X(0)\right]=0,\notag\\
    &\E\left[\frac{\partial}{\partial v_i}X(t)X(t)\right]=\E\left[\frac{\partial}{\partial t_i}X(0)X(0)\right]=0\notag,
  \end{align}
  as well as
  \begin{align*}
    &\E\left[\frac{\partial^2}{\partial v_i\partial v_j}X(t)\frac{\partial^2}{\partial v_k\partial v_l}X(t)\right]= \E\left[\frac{\partial^2}{\partial t_i\partial t_j}X(0)\frac{\partial^2}{\partial t_k\partial t_l}X(0)\right],\\
    &\E\left[\frac{\partial^2}{\partial v_i\partial v_j}X(t)X(t)\right]=\E\left[\frac{\partial^2}{\partial t_i\partial t_j}X(0)X(0)\right],\\
    &\E\left[X(t)X(t)\right]=\E\left[X(0)X(0)\right].
  \end{align*}
  Assumption \ref{as:nondegenerate} yields that $\Sigma$ is positive definite. Hence  the well-known Cholesky decomposition, cf. \cite [Fact 8.9.37]{book:Bernstein}, yields the assertion. \qedhere
\end{proof}

Using $\Lambda$, we define the decorrelated process
\begin{align}\label{def:Y}
  Y^F(t):=\Lambda^{-1}\mathcal{X}^F(t), \quad t\in \R^d, F\in G(d,d-m).
\end{align}
For fixed $t\in \R^d$ and $F\in G(d,d-m)$, the random vector $Y^F(t)$ is standard normal, i.e. $Y^F(t)\sim \mathcal{N}(0,I_{D\times D})$. However, note that for different $t,s\in \R^d$ the vectors $Y^F(t)$ and $Y^F(s)$ are in general not independent. In what follows we shall be using the stationarity
\begin{align*}
  \left (Y^F(t),Y^{F'}(t')\right )=\left ( Y^F(t+h),Y^{F'}(t'+h) \right),
\end{align*}
where $t,t',h \in \R^d$ and $F,F'\in G(d,d-m)$. Indeed, we have for suitable mappings $f^F$ and $f^{F'}$ that
\begin{align*}
(Y^F&(t),Y^{F'}(t'))\\
  &=(f^F(\nabla X(t),D^2X(t),X(t)),f^{F'}(\nabla X(t'),D^2X(t'),X(t')))\\
  &\stackrel{\mathcal{D}}{=}(f^F(\nabla X(t+h),D^2X(t+h),X(t+h)),f^{F'}(\nabla X(t'+h),D^2X(t'+h),X(t'+h)))\\
  &=(Y^F(t+h),Y^{F'}(t'+h)).
\end{align*}

We now define the mapping $G_\varepsilon\colon\R^{d-m}\times \R^{(d-m)(d-m+1)/2+1} \to \R $, where we use the notation $(x)_{i_1,\ldots,i_k}:=(x_{i_1},\ldots, x_{i_k})$, by
\begin{align*}
  G_{\varepsilon}(x,y):= (-1)^{d-m}\delta_{\varepsilon}(x) \det\Big(\big(\Lambda_2y\big)_{1,\ldots,(d-m)(d-m+1)/2} \Big)\ind\{\big(\Lambda_2y\big)_{(d-m)(d-m+1)/2+1}\geq u\},
\end{align*}
so that, by Lemma \ref{lem:repcoord}, we can rewrite the random variable $\zeta^\varepsilon_{m,W}$ as
\begin{align*}
  \zeta^\varepsilon _{m,W} = \int_{G(d,d-m)}\int_{W}G_\varepsilon(Y^F(t))\, dt\, \mu(dF).
\end{align*}
In the above definition the vector $\big(\Lambda_2y\big)_{1,\ldots,(d-m)(d-m+1)/2}$
is identified with the symmetric $(d-m)\times(d-m)$-matrix, whose diagonal and upper diagonal entries are given by $\left(\Lambda_2y\right)_{1,\ldots,(d-m)(d-m+1)/2}$, according to the way one identifies $\left(\frac{\partial^2}{\partial v_i \partial v_j} X(t)\right)_{1\leq i \leq j \leq d-m}$ with a vector.
Moreover the mapping $G_\varepsilon$ is an element of $L^2(\R^D,\phi_D\lambda^D)$, where $\phi_D$ denotes the density of a $D$-dimensional standard normal distribution and $\lambda^D$ the $D$-dimensional Lebesgue measure, and therefore can be expanded in the orthonormal basis $\{n!^{-1/2}\widetilde H_n:n\in \N^D\}$, where $\widetilde H_n:=\otimes_{i=1}^DH_{n_i}$ and $H_{k}(x):=(-1)^ke^{\frac{x^2}{2}}\frac{\partial^k}{\partial x^k}e^{-\frac{x^2}{2}}$, $k\in \N\setminus \{0\}$ and $H_0=1$, cf. \cite [Proposition 1.4.2 (iv)]{book:NourdinPeccati}, \cite [Example E.9]{book:Janson}. Thus we  obtain
\begin{align}\label{eq:L^2approxG}
  G_\varepsilon = \sum_{q=0}^\infty\sum_{n\in \N^D,|n|=q} c(G_\varepsilon,n)\widetilde H_n,
\end{align}
in $L^2(\phi_D\lambda^D)$, where
\begin{align}\label{def:c(Gvarepsilon,n)}
c(G_\varepsilon,n)&:=n!^{-1}\int_{\R^D}G_\varepsilon(x)\widetilde H_n(x) \phi_D(x)\, dx\notag\\
  &=\frac{(-1)^{d-m}}{\prod_{i=1}^{d} n_i!}\int_{\R^{d-m}}\delta_\varepsilon(x)\prod_{i=1}^{d-m} H_{n_i}(x) \phi_{d-m}(x)\,dx\int_{\R^{D-(d-m)}}\ind\{(\Lambda_2y)_{D-(d-m)}\geq u\} \notag\\
  &\quad \times\det\big((\Lambda_2y)_{1,\ldots,(d-m)(d-m+1)/2} \big)\prod_{i=d-m+1}^D H_{n_i}(y) \phi_{D-(d-m)}(y)\,dy.
\end{align}
It is this expansion, which helps to establish an expansion of the random variable $\zeta^\varepsilon_{m,W}$, as is shown in the next lemma.

\begin{lemma}\label{lem:l^2approx}
  Let $\varepsilon>0$ and $W\subset \R^d$ be bounded and convex. Then
  \begin{align*}
    \zeta^\varepsilon_{m,W}=\sum_{q\geq 0} \sum_{n\in \N^D,|n|=q} \int_{G(d,d-m)}c(G_{\varepsilon},n)\int_{W}\widetilde H_n(Y^F(t))\,dt\, \nu(dF),
  \end{align*}
  where the convergence is in $L^2(\mathbb{P})$.
\end{lemma}

\begin{proof}
  The right side is in $L^2(\mathbb{P})$ since it is a Cauchy sequence, which can be seen by Jensen's inequality and \eqref{eq:L^2approxG}. Recall that by Lemma \ref{lem:repcoord}
  \begin{align*}
    \zeta^\varepsilon_{m,W}&=\int_{G(d,d-m)}\int_{W}G_\varepsilon(Y^F(t))\,dt\, \nu(dF),
  \end{align*}
  thus for $Q\in\N$ we have that
  \begin{align*}
    \E\bigg[ &\bigg( \zeta_{m,W}^\varepsilon- \int_{G(d,d-m)}\int_{W}\sum_{q= 0}^Q \sum_{n\in \N^D,|n|=q}c(G_\varepsilon,n)\widetilde H_n(Y^F(t))\,dt\, \nu(dF)\bigg)^2\bigg]\\
    &=\E\bigg[ \bigg( \int_{G(d,d-m)}\int_{W}G_\varepsilon(Y^F(t))-\sum_{q= 0}^Q \sum_{n\in \N^D,|n|=q} c(G_\varepsilon,n)\widetilde H_n(Y^F(t))\,dt\, \nu(dF)\bigg)^2\bigg]\\
    &\leq c\E\bigg[ \int_{G(d,d-m)}\int_{W}\bigg (G_\varepsilon(Y^F(t))-\sum_{q= 0}^Q \sum_{n\in \N^D,|n|=q}c(G_\varepsilon,n)\widetilde H_n(Y^F(t))\bigg)^2 \,dt\, \nu(dF)\bigg],
  \end{align*}
  where we used Jensen's inequality in the last step and $c=\binom{d}{d-m}\frac{\omega_{d}}{\omega_{m}\omega_{d-m}}\mathcal{H}^d(W)$. By Fubini's theorem the latter term equals
  \begin{align*}
    &c \int_{G(d,d-m)}\int_{W}\E\bigg[\bigg(G_\varepsilon(Y^F(t))-\sum_{q= 0}^Q \sum_{n\in \N^D,|n|=q}c(G_\varepsilon,n)\widetilde H_n(Y^F(t))\bigg)^2\bigg] \,dt\, \nu(dF)\\
    &=c \int_{G(d,d-m)}\int_{W}\int_{\R^D}\bigg(G_\varepsilon(x)-\sum_{q= 0}^Q \sum_{n\in \N^D,|n|=q}c(G_\varepsilon,n)\widetilde H_n(x)\bigg)^2 \phi_D(x) \,dx \,dt\, \nu(dF)\\
    &=c^2\int_{\R^D}\bigg(G_\varepsilon(x)-\sum_{q= 0}^Q \sum_{n\in \N^D,|n|=q}c(G_\varepsilon,n)\widetilde H_n(x)\bigg)^2 \phi_D(x) \,dx.
  \end{align*}
  Hence, by (\ref{eq:L^2approxG}), we conclude
  \begin{align*}
    c^2\int_{\R^D}\bigg(G_\varepsilon(x)-\sum_{q= 0}^Q \sum_{n\in \N^D,|n|=q}c(G_\varepsilon,n)\widetilde H_n(x)\bigg)^2 \phi_D(x)\, dx \stackrel{Q\to \infty}{\longrightarrow }0,
  \end{align*}
  which shows the assertion. \qedhere
\end{proof}

The following lemma  is a special case of \cite [Lemma 3.2]{paper:Taqqu}. We give a prove for completeness.
\begin{lemma} \label{lem:generalisedMehlersFormula}
  Let $F,F'\in G(d,d-m), t,t'\in \R^d$ and $n,n'\in \N^D$. Then
  \begin{align*}
    \E[\widetilde H_n(Y^F(t))\widetilde H_{n'}(Y^{F'}(t'))]=\sum_{\substack{d\in \N^{D\times D}, \\ \sum_{i=1}^D d_{ij}=n_j\,,\,\sum_{j=1}^D d_{i j}=n'_i}} n!n'! \prod_{1\leq i,j \leq D}\frac{\E\left[ Y_i^F(t)Y_j^{F'}(t')\right]^{d_{ij}}}{d_{ij}!}
  \end{align*}
  for $|n|=|n'|$ and for $|n|\neq |n'|$
  \begin{align*}
    \E[\widetilde H_n(Y^F(t))\widetilde H_{n'}(Y^{F'}(t'))]=0.
  \end{align*}
\end{lemma}

\begin{proof}
  We first proof the following: Let $V,W$ be two $D$-dimensional random vectors where $(V,W)\sim \mathcal{N}_{2D}\left(0,\begin{pmatrix}
                                       I_{D} & (\E\left[ V_iW_j\right])_{1\leq i,j\leq D} \\
                                       (\E\left[W_i V_j\right])_{1\leq i,j\leq D} & I_{D} \\
                                     \end{pmatrix}
  \right)$. Then
  \begin{align*}
    \E\left[ \widetilde H_n(V)\widetilde H_{n'}(W)\right]=\ind\{|n|=|n'|\}\sum_{\substack{d\in \N^{D\times D}\\ \sum_{i=1}^D d_{ij}=n_j\,,\,\sum_{j=1}^D d_{i j}=n'_i}} n!n'! \prod_{1\leq i,j \leq D}\frac{\E\left[ V_iW_j\right]^{d_{ij}}}{d_{ij}!}.
  \end{align*}
   Observe that via the moment generating function of a multivariate normal distribution, we obtain for $t\in \R^{2D}$
  \begin{align}
    \E&\left[ \prod_{i=1}^D\exp(t_iV_i-\frac{1}{2}t_i^2)\prod_{i=D+1}^{2D}\exp(t_iW_{i-D}-\frac{1}{2}t_i^2)\right] =\exp\left(\sum_{i,j=1}^D t_i t_{D+j}\E\left[ V_iW_j\right]\right).\label{eq:VW}
  \end{align}
  We use the identity $\exp(tx-1/2t^2)=\sum_{q=0}^\infty t^q/q! H_q(x)$ to see the equality of the left side in (\ref{eq:VW}) to
  \begin{align*}
    \sum_{n_1,\ldots ,n_D,n'_1,\ldots,n'_{D}=0}^\infty \frac{t_1^{n_1}\ldots t_D^{n_D}t_1^{n'_1}\ldots t^{n'_D}_D}{n!n'!}\E\left[ \widetilde H_{n}(V)\widetilde H_{n'}(W)\right],
  \end{align*}
  where we used \cite [Lemma 3.1]{paper:Taqqu} to change the order of summation and expectation. The right side in (\ref{eq:VW}) equals
  \begin{align*}
    &\sum_{r=0}^\infty \frac{1}{r!}\left(\sum_{i,j=1}^Dt_it_{D+j}\E\left[ V_iW_j\right]\right)^r
    \\&=\sum_{r=0}^\infty\sum_{d\in\N^{D\times D}, \sum_{i,j=1}^{D} d_{ij}=r} \prod_{1\leq i,j\leq D}\frac{1}{d_{ij}!}(t_it_{D+j})^{d_{ij}}\E\left[ V_iW_j\right]^{d_{ij}}\\
    &=\sum_{r=0}^\infty\sum_{d\in\N^{D\times D}, \sum_{i,j=1}^{D} d_{ij}=r} \prod_{1\leq i,j\leq D}\left(\frac{\E\left[ V_iW_j\right]^{d_{ij}}}{d_{ij}!}\right) t_1^{\sum_{k=1}^D d_{k1}}\ldots t_D^{\sum_{k=1}^D d_{kD}}t_{D+1}^{\sum_{k=1}^D d_{1k}}\ldots t_{2D}^{\sum_{k=1}^D d_{Dk}},
  \end{align*}
  by the multinomial theorem in the first line. Note that the sum over the exponents of the variables  $t_1,\ldots,t_D$ equals the one over the exponents of variables $t_{D+1},\ldots ,t_{2D}$, i.e.
    $\sum_{i=1}^D \sum_{j=1}^Dd_{ji}= \sum_{i=1}^D \sum_{j=1}^Dd_{ij} = r.$
  Hence by comparing the coefficients, we obtain for $|n|\neq |n'|$
  \begin{align*}
    \E\left[ \widetilde H_{n}(V)\widetilde H_{n'}(W)\right] =0,
  \end{align*}
  and furthermore for $|n|=|n'|$, the monomial of degree $(n,n')$ corresponds to $r=\frac{1}{2}(|n|+|n'|)$ and can therefore be found in a unique term of the sum over $r$, which yields the assertion.

  To conclude the lemma, note that the process $(Y^F(t))_{(F,t)\in G(d,d-m)\times \R^d}$ is Gaussian and the vector $Y^F(t)$ is standard normal for fixed $(F,t)\in G(d,d-m)\times \R^d$. \qedhere
\end{proof}

Using the last lemmata, we can now give a Hermite type expansion of the $m$-th Lipschitz-Killing curvature of the excursion set in the ball of radius N, namely $\zeta_{m,N}$. We first define
\begin{align}\label{def:coefc}
    c(n)&:=(2\pi)^{-(d-m)/2}\prod_{i=1}^{d-m}\frac{H_{n_i}(0)}{n_i!}\frac{(-1)^{d-m}}{\prod_{i=d-m+1}^Dn_i!} \int_{\R^{D-(d-m)}}\det\big((\Lambda_2y)_{1,\ldots,(d-m)(d-m+1)/2} \big) \notag\\
  &\quad\times\ind\{(\Lambda_2y)_{(d-m)(d-m+1)/2+1}\geq u\}\prod_{i=d-m+1}^D H_{n_i}(y) \phi_{D-(d-m)}(y)\,dy.
\end{align}
Since $\frac{1}{\prod_{i=1}^{d-m} n_i!}\int_{\R^{d-m}}\delta_\varepsilon(x)\prod_{i=1}^{d-m} H_{n_i}(x) \phi_{d-m}(x)\,dx  \stackrel{\varepsilon\to 0}{\longrightarrow} (2\pi)^{-(d-m)/2}\prod_{i=1}^{d-m}\frac{H_{n_i}(0)}{ n_i!}$ we obtain $c(n)=\lim_{\varepsilon\to 0} c(G_\varepsilon,n)$. The coefficient $c(\cdot)$ is the coefficient in the expansion of $\zeta_{m,W}$ as we see in the next lemma. Note that, the following expansion is orthogonal due to the last lemma.
\begin{theorem}\label{thm:L^2Approx}
  Let $W\subset \R^d$ be convex and bounded. Then
  \begin{align}
    \zeta_{m,W}\stackrel{L^2(\mathbb{P})}{=}\sum_{q\geq 0} \sum_{n\in \N^D,|n|=q} \int_{G(d,d-m)}c(n)\int_{W}\widetilde H_n(Y^F(t))\,dt\, \nu(dF). \label{eq:L^2Approx}
  \end{align}
\end{theorem}

\begin{proof}
  We show that $(\sum_{q=0}^Q\sum_{|n|=q}\int_{G(d,d-m)}c(n)\int_{W}\widetilde H_n(Y^F(t))dt \nu(dF))_{Q\in \N}$ is a Cauchy sequence, so that the right side of the asserted equation is in $L^2(\mathbb{P})$. For $Q_1<Q_2 \in \N$ we have
  \begin{align*}
    \E[&\big( \sum_{q=Q_1}^{Q_2}\sum_{|n|=q}\int_{G(d,d-m)}c(n)\int_{W}\widetilde H_n(Y^F(t))\,dt\, \nu(dF) \big)^2]
    \\&\leq\liminf_{\varepsilon \to 0}\E[\big( \sum_{q=Q_1}^{Q_2}\sum_{|n|=q}\int_{G(d,d-m)}c(G_\varepsilon,n)\int_{W}\widetilde H_n(Y^F(t))\,dt \,\nu(dF) \big)^2]
  \end{align*}
  by Fatou's lemma. The orthogonality of Lemma \ref{lem:generalisedMehlersFormula} yields equality to
  \begin{align}\label{eq:CauchySeq}
    \liminf_{\varepsilon \to 0}&\sum_{q=Q_1}^{Q_2}\E[\big( \sum_{|n|=q}\int_{G(d,d-m)}c(G_\varepsilon,n)\int_{W}\widetilde H_n(Y^F(t))\,dt \,\nu(dF)\big)^2]
    \notag\\&\leq\sum_{q=Q_1}^{\infty}\liminf_{\varepsilon \to 0}\E[\big( \sum_{|n|=q}\int_{G(d,d-m)}c(G_\varepsilon,n)\int_{W}\widetilde H_n(Y^F(t))\,dt\, \nu(dF)\big)^2],
  \end{align}
  where we added positive terms in the second line. Note that, in order to use the orthogonality we need Fubini's theorem, which is applicable as a consequence of \cite [Lemma 3.1]{paper:Taqqu}.
  By Fatou's lemma and the Pythagorean identity the latter is bounded from above by
  \begin{align*}
    \liminf_{\varepsilon \to 0}\E[\big(\sum_{q=0}^{\infty} \sum_{|n|=q}\int_{G(d,d-m)}c(G_\varepsilon,n)\int_{W}\widetilde H_n(Y^F(t))\,dt\, \nu(dF)\big)^2]
    &=\liminf_{\varepsilon \to 0} \E[(\zeta_{m,W}^\varepsilon)^2]\\
    &=\E[(\zeta_{m,W})^2]<\infty,
  \end{align*}
  where we have used Lemma \ref{lem:l^2approx}  and finally Lemma \ref{lem:hilfL^2} (i). Thus \eqref{eq:CauchySeq} is the tail of a convergent series, which yields that the sequence is Cauchy.

  Now define $\widetilde I_q:=\sum_{n\in\N^D,|n|=q}\int_{G(d,d-m)}c(n)\int_{W}\widetilde H_n(Y^F(t))\,dt \,\nu(dF)$ and write $\pi^Q(f)$ for the projection onto the first $Q$ chaos in $L^2(\mathbb{P})$ and likewise $\pi_Q(f)$ for the projection onto the chaos greater than $Q$, $Q\in \N$, $f\in L^2(\mathbb{P})$. To show the asserted equality, observe that
  \begin{align*}
    \|\zeta_{m,W}-\sum_{q=0}^\infty \widetilde I_{q}\|_{L^2}&\leq \|\pi_Q(\zeta_{m,W})-\sum_{q=Q}^\infty\widetilde I_{q}\|_{L^2}+\|\pi^Q(\zeta_{m,W}-\zeta^\varepsilon_{m,W})\|_{L^2}+ \|\pi^Q(\zeta^\varepsilon_{m,W})-\sum_{q=0}^Q\widetilde I_{q}\|_{L^2}\\
    &\leq \|\pi_Q(\zeta_{m,W})\|_{L^2} + \|\sum_{q=Q}^\infty\widetilde I_{q}\|_{L^2} + \|\zeta_{m,W}-\zeta^\varepsilon_{m,W}\|_{L^2}+ \|\pi^Q(\zeta^\varepsilon_{m,W})-\sum_{q=0}^Q\widetilde I_{q}\|_{L^2}.
  \end{align*}
  The first two terms tend to 0 for $Q\to \infty$, since both functions belong to $L^2(\mathbb{P})$, as does the third one for $\varepsilon \to 0$, due to Lemma \ref{lem:L^2approx}. For the last one we have
  \begin{align*}
    \|\pi^Q(\zeta^\varepsilon_{m,W})-\sum_{q=0}^Q\widetilde I_{q}\|_{L^2}&=\E[\big(\sum_{q=0}^Q\sum_{|n|=q}\int_{G(d,d-m)}c(G_\varepsilon,n)\int_{W}\widetilde H_n(Y^F(t))\,dt \,\nu(dF)\\
    &\quad-\sum_{q=0}^Q\sum_{|n|=q}\int_{G(d,d-m)}\lim_{\varepsilon \to 0}c(G_\varepsilon,n)\int_{W}\widetilde H_n(Y^F(t))\,dt \,\nu(dF) \big)^2],
  \end{align*}
  which equals
  \begin{align*}
   &\sum_{q,q'=0}^Q\sum_{|n|=q}\sum_{|n'|=q'}\left (c(G_\varepsilon,n)-\lim_{\varepsilon \to 0}c(G_\varepsilon,n)\right)\left (c(G_\varepsilon,n')-\lim_{\varepsilon \to 0}c(G_\varepsilon,n')\right)\\
    &\quad \times \E[\int_{G(d,d-m)}\int_{W}\widetilde H_n(Y^F(t))\,dt\, \nu(dF)\int_{G(d,d-m)}\int_{W}\widetilde H_{n'}(Y^F(t))\,dt\, \nu(dF)].
  \end{align*}
  The assertion follows by first taking the limit $\varepsilon \to 0$ and then $Q\to \infty$. \qedhere
\end{proof}

\section{Embedding into an isonormal process}
We now embed the Gaussian field $(Y^F_t\colon \Omega \to \R^D \mid (F,t)\in G(d,d-m)\times \R^d)$ into an isonormal process. By standard theory, for instance in \cite [Section 5.4]{book:AdlerTaylor}, we obtain for $s,t\in \R^d$
\begin{align*}
  \on{Cov}^X(s,t)=\on{Cov}^X(s-t)=\int_{\R^d}e^{i\langle s-t, \lambda \rangle} \,f\lambda^d(d\lambda),
\end{align*}
where $f\lambda^d$ denotes the spectral measure of $X$ and $f$ the spectral density. Recall that the spectral density exists due to \ref{as:covariance}. Moreover we obtain
\begin{align}\label{eq:covDif}
  \E&\left[ \frac{\partial^k}{\partial v_{1}\ldots\partial v_{k}} X (t)\frac{\partial^l}{\partial v'_{1}\ldots\partial v'_{l}} X (s)\right]=(-1)^l \int_{\R^d}\frac{\partial^{(k+l)}}{\partial v_{1}\ldots\partial v_{k}\partial v'_{1}\ldots\partial v'_{l}}\left(e^{i\langle \cdot, \lambda \rangle}\right)(s-t)\,f\lambda^d(d\lambda),
\end{align}
where $k,l\in \{0,1,2\}$,  $v_1,\ldots, v_{k},v'_1,\dots, v'_l\in S^{d-1}$. We define the real  Hilbert space \begin{align*}
  \mathfrak{H}:=\{h\colon \R^d\to \C\mid h(-x)=\overline{h(x)}\}
\end{align*}
equipped with the scalarproduct $\langle f,g\rangle_{L^2(f\lambda^d)}:=\int_{\R^d}f(\lambda)\overline{g(\lambda)}\,f\lambda^d(d\lambda)$, which is real since the functions are Hermitian and $f\lambda^d$ is symmetric. By \cite [Prop. 2.1.1]{book:NourdinPeccati}, we know that there exists an isonormal process $W$ on $\mathfrak{H}$, so that for $f,g\in \mathfrak{H}$
\begin{align}\label{eq:isoProcess}
  \E\left[ W(f)W(g)\right]=\langle f,g\rangle_{L^2(f\lambda^d)}.
\end{align}

Moreover we define for $F\in G(d,d-m)$ and $j=1,\ldots,D$ the mapping
\begin{align*}
  \varphi^F_{t,j}\colon \R^d \to \C, \lambda \mapsto \sum_{k=1}^D \Lambda_{jk}^{-1}\nu_{k}^F(\lambda)e^{i\langle t, \lambda \rangle} \in \mathfrak{H},
\end{align*}
where
\begin{align*}
  \nu^F\colon \R^d\to \C^D, \lambda \mapsto \big((i\langle v_l,\lambda\rangle)_{1\leq l \leq d-m},(-\langle v_l,\lambda\rangle \langle v_s,\lambda\rangle)_{1\leq l\leq s\leq d-m},1\big )
\end{align*}
and $v_1,\ldots, v_{d-m}$ denotes the chosen orthonormal basis of $F$. Note that $\nu_k^F(\lambda)e^{i\langle \cdot, \lambda \rangle}$ is the directional derivative of $e^{i\langle \cdot, \lambda \rangle}$ of the same order and in the same direction as the derivative of $X$ in the $k$-th component of $\mathcal{X}^F$.

Then we obtain
\begin{align*}
  Y^\cdot(\cdot\cdot)\stackrel{\mathcal{D}}{=} \left(W(\varphi^\cdot_{\cdot\cdot,1}),\ldots, W(\varphi^\cdot_{\cdot\cdot,D})\right)
\end{align*}
as processes on $G(d,d-m)\times \R^d$. To see this, it suffices to show that their covariance structures coincide, since both processes are centered Gaussian processes. By the definition of $Y$, cf. \eqref{def:Y}, and \eqref{eq:covDif}
\begin{align}\label{eq:covY}
  \E[Y^F_i(t)Y^{F'}_j(t')]&= \sum_{r,s=1}^{D}\Lambda_{ir}^{-1}\Lambda_{js}^{-1} \E[\mathcal{X}_{r}^F(t)\mathcal{X}_{s}^{F'}(t')]\notag \\
  &= \sum_{r,s=1}^{D}\Lambda_{ir}^{-1}\Lambda_{js}^{-1}\int_{\R^d} \nu^F_r(\lambda)e^{i\langle t, \lambda \rangle}\overline{\nu^{F'}_s(\lambda)e^{i\langle t', \lambda \rangle}} \,f\lambda^d(d\lambda)\notag \\
  &=\langle \varphi^F_{t,i},\varphi^{F'}_{t',j}\rangle_{L^2(f\lambda^d)},
\end{align}
for $(F,t),(F',t')\in G(d,d-m)\times \R^d$ and $i,j\in\{1,\ldots, D\}$. By \eqref{eq:isoProcess} we obtain
\begin{align*}
  \langle \varphi^F_{t,i},\varphi^{F'}_{t',j}\rangle_{L^2(f\lambda^d)} =\E[W(\varphi^F_{t,i})W(\varphi^{F'}_{t',j})]
\end{align*}
and therefore the assertion.  Moreover, observe that
\begin{align*}
    \langle \varphi^F_{t,i},\varphi^F_{t,j}\rangle_{L^2(f\lambda^d)} =\E[Y^F_i(t)Y^F_j(t)]=\delta_{ij},
\end{align*}
for $i,j=1,\ldots,D$ and $(F,t)\in G(d,d-m)\times\R^d$. Hence \cite [Theorem 13.25]{book:Kallenberg} implies the second equality in
\begin{align*}
  \prod_{i=1}^D H_{n_i}(Y^\cdot_i(\cdot\cdot))\stackrel{\mathcal{D}}{=}\prod_{i=1}^DH_{n_i}(W(\varphi^\cdot_{\cdot\cdot,i}))=I_q(\varphi_{\cdot\cdot,1}^{\cdot\otimes n_1}\otimes \ldots \otimes \varphi_{\cdot\cdot,D}^{\cdot\otimes n_D}),
\end{align*}
where $q,D\in \N$ and $n\in \N^D$ such that $|n|=q$.
The last equation and Theorem \ref{thm:L^2Approx} with the choice $W=B^d_N$ yield
\begin{align*}
  \frac{\zeta_{m,B^d_N}-\E[\zeta_{m,B^d_N}]}{(N^{d}\kappa_{d})^{1/2}}\stackrel{\mathcal{D}}{=}& \sum_{q=1}^\infty\frac{1}{(N^{d}\kappa_{d})^{1/2}}\sum_{n\in\N^D,|n|=q}\int_{G(d,d-m)} c(n)\int_{B^d_N} I_{q}(\varphi_{t,1}^{F\otimes n_1}\otimes \ldots \otimes \varphi_{t,D}^{F\otimes n_D}) \,dt \,\nu(dF),
\end{align*}
where the right side converges in $L^2(\mathbb{P})$.
We now symmetrise the arguments of the stochastic integral. To this end define for $q,D\in \N$ and $n\in \N^D$ with $|n|=q$ the set
\begin{align*}
  \A_n:=\{k\in\{1,\ldots,D\}^q:\sum_{j=1}^q\ind_{\{i\}}(k_j)=n_i,\forall i=1,\ldots,D\}
\end{align*}
of multiindices, which contain the number $i$ exactly $n_i$ times. Note that for $k\in \A_n$ all permutations of $k$ are also in $\A_n$, and moreover, these sets form a partition of the set $\{1,\ldots,D\}^q$, i.e. $\{1,\ldots,D\}^q=\dot\cup_{n\in\N^D,|n|=q}\A_n$. We further define for $k\in \{1,\ldots,D\}^q$
\begin{align*}
  b(k):= \sum_{n\in\N^D,|n|=q}\ind\{k\in \A_n\}\frac{c(n)}{|\A_n|},
\end{align*}
which is symmetric in the components of $k$.
Since the Wiener-It\^o integrals are invariant with respect to permutations, we obtain for $n\in \N^D$ with $|n|=q$
\begin{align*}
  I_q(\varphi_{t,1}^{F\otimes n_1}\otimes \ldots \otimes \varphi_{t,D}^{F\otimes n_D}) = \frac{1}{|\A_n|}\sum_{k\in \A_n}I_q(\varphi^F_{t,k_1}\otimes \ldots \otimes \varphi^F_{t,k_q})
\end{align*}
and thus
\begin{align*}
  \sum_{n\in\N^D,|n|=q} c(n)I_q(\varphi_{t,1}^{F\otimes n_1}\otimes \ldots \otimes \varphi_{t,D}^{F\otimes n_D})&= \sum_{n\in\N^D,|n|=q}\sum_{k\in\A_n}\frac{c(n)}{|\A_n|}I_q(\varphi_{t,k_1}^F\otimes \ldots \otimes \varphi^F_{t,k_q})\\
  &=\sum_{k\in\{1,\ldots,D\}^q}b(k)I_q(\varphi^F_{t,k_1}\otimes \ldots \otimes \varphi^F_{t,k_q}).
\end{align*}
Hence by Fubini's theorem for Wiener-It\^o integrals, we finally obtain a representation for the standardized $\zeta_{m,B^d_N}$, which is amenable to the theory described in \cite{book:NourdinPeccati}, i.e.
\begin{align*}
  \frac{\zeta_{m,B^d_N}-\E[\zeta_{m,B^d_N}]}{(N^{d}\kappa_{d})^{1/2}}\stackrel{\mathcal{D}}{=}\sum_{q=1}^\infty I_q(g_{N,q}),
\end{align*}
where
\begin{align}\label{def:gNq}
  g_{N,q}:=\frac{1}{(N^{d}\kappa_{d})^{1/2}}\sum_{k\in\{1,\ldots,D\}^q}\int_{G(d,d-m)} b(k)\int_{B^d_N} \varphi^F_{t,k_1}\otimes \ldots \otimes \varphi^F_{t,k_q} \,dt\, \nu(dF)
\end{align}
is symmetric, since the coefficients $b(\cdot)$ are.

\section{Proof of the main theorem}
We now apply Theorem 6.3.1 in \cite{book:NourdinPeccati}, which yields, once we have checked the required conditions, the main theorem of this paper. We repeat it here for completeness and note that in the monograph \cite{book:NourdinPeccati} condition (iv) is stated slightly differently but the proof given there remains the same.

\begin{theorem}[Theorem 6.3.1 in \cite{book:NourdinPeccati}]\label{thm:NourdinPeccati}
Let  $F_N \in L^2(\mathbb{P})$, for $N\in \N$, such that $\mathbb{E}[F_N]=0$. Then there exist functions $g_{N,q}\in \mathfrak{H}^{\odot q}$, for $N,q\in \N$, such that $F_N=\sum_{q\geq 1} I_q(g_{N,q})$.
Suppose that the following conditions
\begin{enumerate}[label={\normalfont(\roman*)}]
  \item For fixed $q\geq 1$ there exists $\sigma^2_q\geq 0$ such that $q!\|g_{N,q}\|^2_{\mathfrak{H}^{\otimes q}} \stackrel{N \rightarrow \infty}{\longrightarrow }\sigma_q^2$,
  \item $\sigma^2:=\sum_{q\geq 1} \sigma_q^2 <\infty$,
  \item For all $q \geq 2$ and $r=1,\ldots,q-1$ we have $\|g_{N,q}\otimes_{r}g_{N,q}\|_{\mathfrak{H}^{\otimes(2q-2r)}}\stackrel{N \rightarrow \infty}{\longrightarrow } 0$,
  \item $\lim_{Q\to \infty} \limsup_{N\to \infty} \sum_{q=Q+1}^\infty q! \|g_{N,q}\|^2_{\mathfrak{H}^{\otimes q}}=0$
\end{enumerate}
are true. Then $F_n \stackrel{\mathcal{D}}{\longrightarrow } \mathcal{N}(0,\sigma^2)$.
\end{theorem}

Before we verify the conditions, we need to prove the following auxiliary lemma, which will be needed for condition (ii) and (iv).
\begin{lemma}\label{lem:auxBounds}
  There exists $c>0$ depending on the covariance of $X$, $d$ and $m$ such that
  \begin{enumerate}[label={\normalfont(\roman*)}]
    \item \label{lem:ub} $\sum_{n\in\N^D,|n|=q}c(n)^2 n! \leq c q^D$ for $q\geq 1$.
    \item \label{lem:uniBound} $\sup\limits_{W\subset [0,1)^d \mathrm{\, convex}} \E\left[ (\sum_{q=1}^\infty \sum_{n\in \N^D, |n|=q}c(n)\int_{G(d,d-m)}\int_{W}\widetilde H_n(Y^F(t))dt \nu(dF))^2\right] \leq c.$
  \end{enumerate}
\end{lemma}

\begin{proof}
  In the following the constant $c>0$ may be changing from line to line.
  Recall \eqref{def:coefc}
  \begin{align*}
    &c(n)=(2\pi)^{-(d-m)/2}\prod_{i=1}^{d-m}\frac{H_{n_i}(0)}{n_i!}\frac{(-1)^{d-m}}{\prod_{i=d-m+1}^Dn_i!} \\
    &\quad\times\underbrace{\int_{\R^{D-d+m}}\det\big((\Lambda_2y)_{1,\ldots,\frac{(d-m)(d-m+1)}{2}} \big)\ind\{(\Lambda_2y)_{D-d+m}\geq u\}\prod_{i=d-m+1}^D H_{n_i}(y) \phi_{D-d+m}(y)\,dy.}_{:=Z(n)}
  \end{align*}
  Proposition 3 in \cite{paper:Imkeller} yields $\prod_{i=1}^{d-m}\frac{|H_{n_i}(0)|}{\sqrt{n_i!}}\leq c$, for a constant $c>0$, and thus
  \begin{align*}
    \left((2\pi)^{-(d-m)/2}\prod_{i=1}^{d-m}\frac{H_{n_i}(0)}{n_i!}\right)^2 \leq \frac{c}{ \prod_{i=1}^{d-m}n_i!}.
  \end{align*}
  By H\"older's inequality, we obtain
  \begin{align*}
    Z(n)^2&\leq \int_{\R^{D-d+m}}\det\big( (\Lambda_2y)_{1,\ldots,\frac{(d-m)(d-m+1)}{2}} \big)^2\ind\{(\Lambda_2y)_{D-d+m}\geq u\}\phi_{D-d+m}(y)\, dy\\
    &\quad\times\int_{\R^{D-d+m}}\left(\prod_{i=d-m+1}^D H_{n_i}(y)\right)^2 \phi_{D-d+m}(y)\,dy\\
    &=c\prod_{i=d-m+1}^D n_i!.
  \end{align*}
  The last two inequalities yield for $q\geq 1$
  \begin{equation*}
    \sum_{n\in \N^D,|n|=q}c(n)^2n! \leq c \sum_{n\in \N^D,|n|=q}1 \leq c \sum_{0\leq n_1,\ldots,n_D\leq q}1 \leq c(q+1)^D\leq cq^D,
  \end{equation*}
  which shows the first assertion. We now proof \ref{lem:uniBound}. By Theorem \ref{thm:L^2Approx}
  \begin{align*}
    \sum_{q=0}^\infty& \sum_{n\in \N^D, |n|=q}c(n)\int_{G(d,d-m)}\int_{W}\widetilde H_n(Y^F(t))\,dt\, \nu(dF) \\
    &=\int_{A(d,d-m)} \#\{t\in W\cap F \mid X(t)\geq u, \nabla (X|_F)(t)=0,\iota_{-X,W\cap F}(t )\text{ even}\}\\
    &\hspace{1.5cm} -\#\{t\in W\cap F \mid X(t)\geq u, \nabla (X|_F)(t)=0,\iota_{-X,W\cap F}(t )\text{ odd}\}\, \mu(dF),
  \end{align*}
  whose second moment is a upper bound for the expectation in \ref{lem:uniBound}. The latter can be bounded by
  \begin{align*}
    \int_{A(d,d-m)} \#\{t\in W\cap F \mid  \nabla (X|_F)(t)=0\}\, \mu(dF).
  \end{align*}
  By Jensen's inequality and Fubini's theorem
  \begin{align*}
    &\E\left[ (\int_{A(d,d-m)} \#\{t\in W\cap F \mid  \nabla (X|_F)(t)=0\}\, \mu(dF))^2 \right]\\
    \quad&\leq \mu(\{F:F\cap W\neq \emptyset\}) \int_{A(d,d-m)}\E\left[\#\{t\in W\cap F \mid  \nabla (X|_F)(t)=0\}^2 \right]\,\mu(dF).
  \end{align*}
  Now, by \eqref{eq:rice} and \eqref{eq:ricefac} we bound the integrand by
  \begin{align*}
    c\mathcal{H}^{d-m}(W^F_{v(F)}) + &\int_{W^F_{v(F)}-W^F_{v(F)}} \E\left[ |\det D^2_{v(F)}(X)(\rho_{v(F)}^F(t))\det D^2_{v(F)}(X)(\rho_{v(F)}^F(0))| \mid \mathcal{E}(F,t,0) \right]\\
    &\times p_{\nabla_{v(F)} (X)(\rho_{v(F)}^F(t))\nabla_{v(F)} (X)(\rho_{v(F)}^F(0))}(0,0)\mathcal{H}^{d-m}(W^F_{v(F)}\cap (W^F_{v(F)}-t)) \,dt,
  \end{align*}
  where $c\geq0$ is a constant depending on $X$, $d$ and $m$. Taking $N:=d^{1/2}$ in Lemmata \ref{lem:boundDensity} and  \ref{lem:boundExp} we obtain for the second summand the upper bound
  \begin{align*}
    c\int_{W^F_{v(F)}-W^F_{v(F)}} \|t\|^{-(d-m)+2}\mathcal{H}^{d-m}(W^F_{v(F)}\cap (W^F_{v(F)}-t))\,dt.
  \end{align*}
  Since $W^F_{v(F)}\subset  B^{d-m}_{d^{1/2}}$ and $W^F_{v(F)}-W^F_{v(F)}\subset B^{d-m}_{d^{1/2}}$, we conclude
  \begin{align*}
    \E\left[\#\{t\in W\cap F \mid \nabla (X|_F)(t)=0\}^2 \right] \leq c\mathcal{H}^{d-m}(B^{d-m}_{d^{1/2}})\left(1+\int_{B^{d-m}_{d^{1/2}}}\|t\|^{-(d-m)+2}\,dt\right).
  \end{align*}
  Hence
  \begin{align*}
    \E&\left[ (\int_{A(d,d-m)} \#\{t\in W\cap F \mid \nabla (X|_F)(t)=0\}\, \mu(dF))^2 \right]\\
    &\quad\leq c\mu(\{F:F\cap B^d_{d^{1/2}}\neq \emptyset\})^2\mathcal{H}^{d-m}(B^{d-m}_{d^{1/2}}) \left(1+\int_{B^{d-m}_{d^{1/2}}}\|t\|^{-(d-m)+2}\,dt\right).\qedhere
  \end{align*}
\end{proof}
In the following we verify the conditions of Theorem \ref{thm:NourdinPeccati}.

{\textbf{Condition (i):}} We calculate the norm  of $g_{N,q}$, cf. \eqref{def:gNq}, by an application of Fubini's theorem and obtain
\begin{align*}
  &q!\| g_{N,q}\|^{2}_{\mathfrak{H}^{\otimes q}}\\
  &= \int_{\R^{dq}}\frac{q!}{N^d\kappa_d}\sum_{k,l\in\{1,\ldots, D\}^q}b(k)b(l)\int_{G(d,d-m)}\int_{G(d,d-m)}\int_{B^d_N}\int_{B^d_N}\varphi^{F}_{t,k_1}\otimes \ldots \otimes \varphi^{F}_{t,k_q}(\lambda_1,\dots,\lambda_q)\\
  &\quad\times\overline{\varphi^{F'}_{t',l_1}}\otimes \ldots \otimes \overline{\varphi^{F'}_{t',l_q}}(\lambda_1,\ldots,\lambda_q) \,dt\,dt'\,\nu(dF)\,\nu(dF')\,(f\lambda^d)^q(d(\lambda_1,\ldots,\lambda_q))\\
  &=\frac{q!}{N^d\kappa_d}\sum_{k,l\in\{1,\ldots, D\}^q}b(k)b(l)\\
  &\quad\times \int_{G(d,d-m)}\int_{G(d,d-m)}\int_{B^d_N}\int_{B^d_N}\prod_{i=1}^q \int_{\R^d}\varphi^{F}_{t,k_i}(\lambda)\overline{\varphi^{F'}_{t',l_i}}(\lambda)\,f \lambda^d(d\lambda) \,dt\, dt'\, \nu(dF)\,\nu(dF').
\end{align*}
Recalling $\int_{\R^d}\varphi^{F}_{t,k_i}(\lambda)\overline{\varphi^{F'}_{t',l_i}}(\lambda)\,f\lambda^d(d\lambda) =\E\left[Y^F_{k_i}(t)Y^{F'}_{l_i}(t')\right]$ in \eqref{eq:covY},
the above equals
\begin{align*}
\frac{q!}{N^d\kappa_d}&\sum_{k,l\in\{1,\ldots, D\}^q}b(k)b(l)\int_{G(d,d-m)}\int_{G(d,d-m)}\int_{B^d_N}\int_{B^d_N}\prod_{i=1}^q \E\left[Y^F_{k_i}(t)Y^{F'}_{l_i}(t')\right] \,dt \,dt' \,\nu(dF)\,\nu(dF').
\end{align*}
By stationarity in the first and Fubini's theorem in the second line
\begin{align*}
  &\frac{1}{N^d\kappa_d}\int_{G(d,d-m)}\int_{G(d,d-m)}\int_{B^d_N}\int_{B^d_N} \prod_{i=1}^q\E\left[Y^F_{k_i}(t)Y^{F'}_{l_i}(t')\right] \,dt\, dt'\,\nu(dF)\,\nu(dF')\\
  &=\frac{1}{N^d\kappa_d}\int_{G(d,d-m)}\int_{G(d,d-m)}\int_{B^d_N}\int_{B^d_N-t'} \prod_{i=1}^q\E\left[Y^F_{k_i}(t+t')Y^{F'}_{l_i}(t')\right] \,dt\, dt'\,\nu(dF)\,\nu(dF')\\
  &=\int_{B^d_{2N}}\int_{G(d,d-m)}\int_{G(d,d-m)} \prod_{i=1}^q\E\left[Y^F_{k_i}(t)Y^{F'}_{l_i}(0)\right]\frac{\mathcal{H}^d((B^d_N-t)\cap B^d_N)}{N^d\kappa_d}\,\nu(dF)\,\nu(dF')\,dt.
\end{align*}
By the definition of $Y$, cf. \eqref{def:Y}, we have the following equality for the covariance matrix of $Y$ $\left(\E\left[Y_i^F(t)Y_j^{F'}(0)\right]\right)_{1\leq i,j\leq D}=\Lambda^{-1}\left(\E\left[\mathcal{X}^F_i(t)\mathcal{X}^{F'}_j(0)\right]\right)_{1\leq i,j\leq D}\Lambda^{-\top}$ and therefore by assumption \ref{as:covariance} there exists a constant $c\geq 0 $ such that
\begin{align*}
  \sup_{1\leq i,j\leq D}\left|\E\left[Y_i^F(t)Y_j^{F'}(0)\right]\right| \leq c\widetilde\psi(t),
\end{align*}
which is an integrable upper bound. By the dominated convergence theorem
\begin{align}\label{eq:asymptoticVariance}
  q!&\| g_{N,q}\|^{2}_{\mathfrak{H}^{\otimes q}}\notag\\
  &=\frac{q!}{N^d\kappa_d}\sum_{k,l\in\{1,\ldots, D\}^q}b(k)b(l)\int_{B^d_{2N}}\int_{G(d,d-m)}\int_{G(d,d-m)}\mathcal{H}^d((B^d_N-t)\cap B^d_N)\notag\\
  &\hspace{2cm}\times\prod_{i=1}^q \E\left[Y^F_{k_i}(t)Y^{F'}_{l_i}(0)\right] \,\nu(dF)\,\nu(dF')\,dt
  \notag\\&  \stackrel{N\to \infty}{\longrightarrow} q!\sum_{k,l\in\{1,\ldots, D\}^q}b(k)b(l)\int_{\R^d}\int_{G(d,d-m)}\int_{G(d,d-m)}\prod_{i=1}^q \E\left[Y^F_{k_i}(t)Y^{F'}_{l_i}(0)\right]\,\nu(dF)\,\nu(dF')\,dt,
\end{align}
where we define the limit as $\sigma^2_{m,q}$. Note that we implicitly used $\mathcal{H}^d((B^d_N-t)\cap B^d_N)/\mathcal{H}^d(B^d_N)\to 1$ for $N\to \infty$ and $t\in \R^d$. To see this consider the discussion following equation (3.22) in \cite{paper:HugLastSchulte}.

\textbf{Condition (ii):} We observe that
\begin{align*}
  \sum_{q=1}^\infty &\lim_{N\to \infty} q!\|g_{N,q}\|^2= \sum_{q=1}^\infty\lim_{N\to \infty} \E\left[ I_q(g_{N,q})^2\right]
\end{align*}
by the It\^o isometry. Reversing  some of the earlier manipulations the latter equals
\begin{align*}
  \sum_{q=1}^\infty\lim_{N\to \infty}\frac{1}{\mathcal{H}^d(B^d_N)}\E\left[\left (\sum_{n\in\N^D,|n|=q}\int_{G(d,d-m)}c(n)\int_{B^d_N}\widetilde H_n(Y^F(t))\,dt\, \nu(dF)\right)^2\right].
\end{align*}
Fatou's Lemma and orthogonality yield the upper bound
\begin{align*}
  \liminf_{N\to \infty}\frac{1}{\mathcal{H}^d(B^d_N)}\E\left[\left (\sum_{q=1}^\infty\sum_{n\in\N^D,|n|=q}\int_{G(d,d-m)}c(n)\int_{B^d_N}\widetilde H_n(Y^F(t))\,dt\, \nu(dF)\right)^2\right].
\end{align*}
Partitioning the space $\R^d$ into translates of the unit cube $[0,1)^d$ the latter without the limit inferior equals
\begin{align*}
\frac{1}{\mathcal{H}^d(B^d_N)}\sum_{z_1,z_2\in \Z^d\cap B^d_{N+d}} &\E\left[ \sum_{q=1}^\infty\sum_{|n|=q}\int_{G(d,d-m)}c(n)\int_{[0,1)^d+z_1}\ind\{t\in B^d_N\} \widetilde H_n( Y^F(t))\,dt\,\nu(dF)\right.\\
&\left.\times \sum_{q=1}^\infty\sum_{|n|=q}\int_{G(d,d-m)}c(n)\int_{[0,1)^d+z_2}\ind\{t\in B^d_N\} \widetilde H_n( Y^F(t))\,dt\,\nu(dF)\right].
\end{align*}
We define $\tau^{F,F'}(t):=\max\{\max_{i}\sum_{k=1}^D|\E\left[ Y^F_i(0)Y^{F'}_k(t)\right]|, \max_{k} \sum_{i=1}^D |\E\left[ Y^F_i(0)Y^{F'}_k(t)\right]| \}$ for $t\in \R^d$, $F,F'\in G(d,d-m)$, and note that due to \ref{as:covariance} there is a constant $c>0$ so that $\tau^{F,F'}(t)\leq c\psi(t)$. Moreover \ref{as:covariance} implies that for $\rho\in (0,1)$ and $\rho<1/c$ there is a constant $s>0$ such that
\begin{align*}
  \psi(t)\leq \rho, \text{ for } |t|\geq s.
\end{align*}
Using $s$, we split the above summation into one over $I_1^N:=\{(z_1,z_2)\in (\Z^d\cap B^d_{N+d})^2\mid \|z_1-z_2\|_\infty \geq s+1 \}$ and $I_2^N:=\{(z_1,z_2)\in (\Z^d\cap B^d_{N+d})^2\mid \|z_1-z_2\|_\infty \leq s \}$. By Fubini's theorem, orthogonality and stationarity the first sum equals
\begin{align}\label{eq:sum1Arc}
  &\sum_{(z_1,z_2)\in I_1^N} \sum_{q=1}^\infty\int_{G(d,d-m)}\int_{G(d,d-m)} \int_{(-1,1)^d}\E\left[ \sum_{|n|=q}c(n) \widetilde H_n(Y^{F}(t+z_1))\sum_{|n|=q}c(n)\widetilde H_n(Y^{F'}(z_2))\right]\notag\\
  &\quad\quad\times \mathcal{H}^d([0,1)^d\cap[0,1)^d-t\cap B^d_N-z_2\cap B^d_N-t-z_1)\,dt\,\nu(dF)\,\nu(dF')\mathcal{H}^d(B^d_N)^{-1}.
\end{align}
Now, we use Lemma 1 in \cite{paper:Arcones}, which reads
\begin{lemma}[Arcones 1994]\label{lem:Arcones}
  Let $V,W$ be two centered $d$-dimensional Gaussian random vectors  such that $\E[V_iV_j]=\E[W_iW_j]=\delta_{ij}$ and let $h\colon \R^d\to \R$ have Hermite rank $r\in \N$ (i.e. $r=\inf\{k\in \N : \exists l_j \text{ such that }\sum_{j=1}^dl_j=k\text{ and }\E[(h(N)-\E[h(N)])\widetilde H_l(N)]\neq 0\}$ where $N\sim\mathcal{N}_d(0,I)$). Define $\tau:=\max\{\max_{1\leq j \leq d}\sum_{k=1}^d |\E[V_jW_k]|,\max_{1\leq k \leq d}\sum_{j=1}^d |\E[V_jW_k]|\}$, which is assumed to be less than 1. Then we have
  \begin{align*}
    |\E[(h(V)-\E[h(V)])(h(W)-\E[h(W)])]|\leq \tau ^r \E[h(V)^2].
  \end{align*}
\end{lemma}
To apply this Lemma, we choose $V=Y^{F'}(z_2)$, $W=Y^F(t+z_1)$ and $h_q\colon \R^D\to \R$ given by $h_q(x):= \sum_{n\in \N^D, |n|=q} c(n) \widetilde H_n(x)$.
Then we have $r=q$ and $\tau^{F,F'} (t+z_1-z_2) \leq c \psi(t+z_1-z_2)<1$ for $t\in (-1,1)^d$ and $z_1,z_2\in I_1^N$. Moreover we have
\begin{align*}
  \E\left[ h_q(Y^{F'}(z_2))^2\right]&=\sum_{n,n'\in\N^D,|n|=|n'|=q} c(n')c(n) \E[\widetilde H_n(Y^{F'}(0))\widetilde H_{n'}(Y^{F'}(0))]\\
  &=\sum_{n,n'\in\N^D,|n|=|n'|=q} c(n')c(n) \prod_{i=1}^D\E[H_{n_i}(Y^{F'}_i(0))H_{n'_i}(Y^{F'}_i(0))]
  \\&=\sum_{n\in\N^D,|n|=q} c(n)^2n!
\end{align*}
and for $q\geq 1$ we obtain $\E\left[ h_q(Y^{F'}(z_2))\right]=\E\left[ h_q(Y^{F}(t+z_1))\right]=0.$
Thus we bound \eqref{eq:sum1Arc} by
\begin{align*}
  K\mathcal{H}^d(B^d_N)^{-1}\sum_{(z_1,z_2)\in I_1^N}\sum_{q=1}^\infty\int_{(-1,1)^d}c^q\psi(t+z_1-z_2)^q\,dt \sum_{n\in\N^D,|n|=q} c(n)^2n!,
\end{align*}
where $K\geq 0$ depends on $d$ and $m$. Lemma \ref{lem:auxBounds}\ref{lem:ub} and $\psi(t+z_1-z_2)^q\leq \rho^{q-1}\psi(t+z_1-z_2)$ yield
\begin{align*}
  K/\rho\mathcal{H}^d(B^d_N)^{-1} \sum_{(z_1,z_2)\in I_1^N} \int_{(-1,1)^d}\psi(t+z_1-z_2)\,dt \sum_{q=1}^\infty q^D (c\rho)^q,
\end{align*}
as an upper bound, where $K\geq 0$ depends on the covariance of $X$, $d$ and $m$. By the estimate $\sum_{(z_1,z_2)\in I_1^N} \int_{(-1,1)^d}\psi(t+z_1-z_2)\,dt\leq 2^d\#\{\Z^d\cap B^d_{N+d}\}\int_{\R^d}\psi(t)\,dt$ we obtain the upper bound
\begin{align*}
  \frac{2^dK}{\rho} \frac{\#\{\Z^d\cap B^d_{N+d}\}}{\mathcal{H}^d(B^d_N)}\int_{\R^d}\psi(t)\,dt \sum_{q=1}^\infty q^D(c\rho)^q.
\end{align*}
The latter is finite since $\liminf_{N\to \infty}\frac{\#\{\Z^d\cap B^d_{N+d}\}}{\mathcal{H}^d(B^d_N)}=1$ and the series converges.

We now analyse the sum over $I_2^N$ and start by using the inequality $ab\leq a^2+b^2$, $a,b\in \R$, to obtain the upper bound
\begin{align*}
  \frac{2(2s+1)^d}{\mathcal{H}^d(B^d_N)}\sum_{z\in \Z^d\cap B^d_{N+d}}\hspace{-3pt}\E\left[ \left(\sum_{q=1}^\infty \sum_{|n|=q}\int_{G(d,d-m)}c(n)\int_{[0,1)^d+z}\ind\{t\in B^d_N\}\widetilde H_n(Y^F(t))\,dt \,\nu(dF)\right)^2\right].
\end{align*}
By stationarity the latter can be bounded by
\begin{align*}
  2(2s+1)^d\frac{\#\{\Z^d\cap B^d_{N+d}\}}{\mathcal{H}^d(B^d_N)}\sup_{\substack{W\subset [0,1)^d\\ \mathrm{\, convex}}}\E\left[ \left(\sum_{q=1}^\infty \sum_{|n|=q}\int_{G(d,d-m)}c(n)\int_{W}\widetilde H_n(Y^F(t))\,dt \,\nu(dF)\right)^2\right],
\end{align*}
whose limit inferior is finite, since the supremum is finite by Lemma \ref{lem:auxBounds}\ref{lem:uniBound}.

\textbf{Condition (iii):}
The $r$-th contraction of $g_{N,q}$, cf. \eqref{def:gNq}, with itself is given by
\begin{align*}
  &g_{N,q}\otimes_r g_{N,q}(a_1,\ldots,a_{2q-2r})\\
  &=\int_{\R^{dr}}\frac{1}{N^d\kappa_d}\sum_{k\in\{1,\ldots,D\}^q}b(k)\int_{G(d,d-m)}\int_{B^d_N} \varphi^F_{t,k_1}(\lambda_1)\ldots\varphi^F_{t,k_r}(\lambda_r)
  \\&\quad\times \varphi^F_{t,k_{r+1}}(a_1)\ldots \varphi^F_{t,k_q}(a_{q-r})\,dt \,\nu(dF) \sum_{l\in \{1,\ldots,D\}^q}b(l)\int_{G(d,d-m)}\int_{B^d_N} \overline{\varphi^{F'}_{t',l_1}}(\lambda_1)\ldots\overline{\varphi^{F'}_{t',l_r}}(\lambda_r)
  \\&\quad\times\varphi^{F'}_{t',l_{r+1}}(a_{q-r+1})\ldots \varphi^{F'}_{t',l_q}(a_{2q-2r})\,dt' \,\nu(dF')\,(f\lambda^d)^r(d(\lambda_1,\ldots,\lambda_r)),
\end{align*}
for $a_1,\ldots, a_{2q-2r}\in \R^d$. By Fubini's theorem and \eqref{eq:covY} the above equals
\begin{align*}
  \frac{1}{N^d\kappa_d}&\sum_{k,l\in \{1,\ldots,D\}^q} b(k)b(l)\int_{G(d,d-m)} \int_{G(d,d-m)}\int_{B^d_N} \int_{B^d_N} \prod_{i=1}^r \int_{\R^d}\varphi^F_{t,k_i}(\lambda)\overline{\varphi^{F'}_{t',l_i}}(\lambda)\,f\lambda^d(d\lambda)
  \\&\quad\times \prod_{i=r+1}^q\varphi^F_{t,k_i}(a_{i-r})\varphi^{F'}_{t',l_i}(a_{q-2r+i})\,dt \,dt'\,\nu(dF)\,\nu(dF')
  \\&=\frac{1}{N^d\kappa_d}\sum_{k,l\in \{1,\ldots,D\}^q} b(k)b(l)\int_{G(d,d-m)} \int_{G(d,d-m)}\int_{B^d_N} \int_{B^d_N} \prod_{i=1}^r \E[Y^F_{k_i}(t)Y^{F'}_{l_i}(t')]
  \\&\quad\times \prod_{i=r+1}^q\varphi^F_{t,k_i}(a_{i-r})\varphi^{F'}_{t',l_i}(a_{q-2r+i})\,dt \,dt'\,\nu(dF)\,\nu(dF').
\end{align*}
Thus we obtain for the norm
\begin{align*}
  &\|g_{N,q}\otimes_r g_{N,q}\|_{\mathfrak{H}^{\otimes(2q-2r)}}^2
  \\&=\int_{\R^{d(2q-2r)}}\frac{1}{N^d\kappa_d}\sum_{k,l\in \{1,\ldots,D\}^q}b(k)b(l)\int_{G(d,d-m)} \int_{G(d,d-m)}\int_{B^d_N} \int_{B^d_N} \prod_{i=1}^r \E[Y^F_{k_i}(t)Y^{F'}_{l_i}(t')]
  \\&\quad\times \prod_{i=r+1}^q\varphi^F_{t,k_i}(a_{i-r})\varphi^{F'}_{t',l_i}(a_{q-2r+i})\,dt \,dt'\,\nu(dF)\,\nu(dF')
  \\&\quad\times \frac{1}{N^d\kappa_d}\sum_{k,l\in \{1,\ldots,D\}^q} b(k)b(l)\int_{G(d,d-m)} \int_{G(d,d-m)}\int_{B^d_N} \int_{B^d_N} \prod_{i=1}^r \E[Y^F_{k_i}(t)Y^{F'}_{l_i}(t')]
  \\&\quad\times \prod_{i=r+1}^q\overline{\varphi^F_{t,k_i}}(a_{i-r})\overline{\varphi^{F'}_{t',l_i}}(a_{q-2r+i})\,dt \,dt'\,\nu(dF)\,\nu(dF')\,(f\lambda^d)^{2q-2r}(d(a_1,\ldots,a_{2q-2r})).
\end{align*}
And again Fubini's theorem yields equality to
\begin{align*}
  \frac{1}{N^{2d}\kappa_d^2}&\sum_{k,l,k',l'\in \{1,\ldots,D\}^q}b(k)b(l)b(k')b(l')\idotsint_{(G(d,d-m))^4}\idotsint_{(B^d_N)^4}
  \\&\quad\times \prod_{i=r+1}^q\int_{\R^d}\varphi^{F_1}_{t_1,k_i}(\lambda) \overline{\varphi^{F_3}_{t_3,k'_i}}(\lambda)\, f\lambda^d(d\lambda) \int_{\R^d}\varphi^{F_2}_{t_2,l_i}(\lambda)\overline{\varphi^{F_4}_{t_4,l'_i}} (\lambda)\,f\lambda^d(d\lambda)
  \\&\quad\times\prod_{i=1}^r\E[Y^{F_1}_{k_i}(t_1)Y_{l_i}^{F_2}(t_2)] \E[Y^{F_3}_{k'_i}(t_3)Y_{l'_i}^{F_4}(t_4)]  \,dt_1\ldots dt_4 \,\nu(dF_1)\ldots \nu(dF_4),
\end{align*}
which by \eqref{eq:covY} equals
\begin{align*}
  \frac{1}{N^{2d}\kappa_d^2}&\sum_{k,l,k',l'\in \{1,\ldots,D\}^q}b(k)b(l)b(k')b(l')
  \\&\quad\times\idotsint_{(G(d,d-m))^4}\idotsint_{(B^d_N)^4} \prod_{i=1}^r\E[Y^{F_1}_{k_i}(t_1)Y_{l_i}^{F_2}(t_2)]\E[Y^{F_3}_{k'_i}(t_3)Y_{l'_i}^{F_4}(t_4)]
  \\&\quad\times\prod_{i=r+1}^q\E[Y^{F_1}_{k_i}(t_1)Y^{F_3}_{k'_i}(t_3)] \E[Y^{F_2}_{l_i}(t_2)Y^{F_4}_{l'_i}(t_4)]\,dt_1\ldots dt_4 \,\nu(dF_1)\ldots \nu(dF_4).
\end{align*}
By \ref{as:covariance} and stationarity there exists a constant $c>0$ such that for all $t\in \R^d$ and $F,F'\in G(d,d-m)$
\begin{align*}
  \sup_{1\leq i, j \leq D}\left | \E\left[Y_i^{F}(t)Y_j^{F'}(s)\right]\right | \leq c \psi(t-s)
\end{align*}
and hence
\begin{align*}
  \|g_{N,q}&\otimes_r g_{N,q}\|_{\mathfrak{H}^{\otimes(2q-2r)}}^2
  \\&\leq \frac{c^{2q}}{N^{2d}\kappa_d^2}\sum_{k,l,k',l'\in \{1,\ldots,D\}^q}b(k)b(l)b(k')b(l')  \idotsint_{(G(d,d-m))^4}\idotsint_{(B^d_N)^4} \psi(t_1-t_2)^r
  \\&\quad\times\psi(t_3-t_4)^{r}\psi(t_1-t_3)^{q-r}\psi(t_2-t_4)^{q-r}\,dt_1\ldots dt_4 \,\nu(dF_1)\ldots \nu(dF_4)
  \\&=c^{2q}\sum_{k,l,k',l'\in \{1,\ldots,D\}^q}b(k)b(l)b(k')b(l')\\
  &\quad\times\underbrace{\frac{1}{N^{2d}\kappa_d^2}\idotsint_{(B^d_N)^4} \psi(t_1-t_2)^r\psi(t_3-t_4)^{r}\psi(t_1-t_3)^{q-r}\psi(t_2-t_4)^{q-r}\,dt_1\ldots dt_4 }_{:=Z(N)}.
\end{align*}
Using the inequality $a^rb^{q-r}\leq a^q+b^q$ for $a=\psi(t_3-t_4)$ and $b=\psi(t_1-t_3)$, we obtain
\begin{align*}
  Z(N)&\leq \frac{1}{N^{2d}\kappa_d^2}\idotsint_{(B^d_N)^4} \psi(t_1-t_2)^r\psi(t_3-t_4)^q\psi(t_2-t_4)^{q-r}\,dt_1\ldots dt_4
  \\&\quad +\frac{1}{N^{2d}\kappa_d^2}\idotsint_{(B^d_N)^4} \psi(t_1-t_2)^r\psi(t_1-t_3)^q\psi(t_2-t_4)^{q-r}\,dt_1\ldots dt_4.
\end{align*}
By \ref{as:covariance} we have $\infty > c_n:= \int_{\R^d}\psi^n(x)\,dx \geq \int_{B^d_N}\psi^n(x)\,dx$, for $n \in \N$, and therefore obtain the following upper bound for the first summand
\begin{align*}
  \frac{1}{N^{2d}\kappa_d^2}&\idotsint_{(B^d_N)^3} \psi(t_1-t_2)^r\psi(t_2-t_4)^{q-r}\int_{\R^d}\psi(t_3-t_4)^q \,dt_3\,dt_1\,dt_2\, dt_4
  \\&\leq\frac{c_q}{N^{2d}\kappa_d^2}\int_{B^d_N}\int_{B^d_N}\psi(t_1-t_2)^r \int_{\R^d}\psi(t_2-t_4)^{q-r} \,dt_4 \,dt_1\,dt_2.
\end{align*}
Repeating this argument yields the upper bound
\begin{align*}
  \frac{c_qc_{q-r}}{N^{2d}\kappa_d^2}\int_{B^d_N}\int_{\R^d}\psi(t_1-t_2)^r \,dt_1\,dt_2=\frac{c_qc_{q-r}c_r}{N^{2d}\kappa_d^2} N^d\kappa_d.
\end{align*}
Note that $\frac{c_qc_{q-r}c_r}{N^{d}\kappa_d} \stackrel{N\to \infty}{\longrightarrow} 0$. Proceeding analogously for the second summand yields
\begin{align*}
  \|g_{N,q}&\otimes_r g_{N,q}\|_{\mathfrak{H}^{\otimes(2q-2r)}}^2 \leq c^{2q}\sum_{k,l,k',l'\in \{1,\ldots,D\}^q}b(k)b(l)b(k')b(l') Z(N) \stackrel{N\to \infty}{\longrightarrow} 0.
\end{align*}

{\textbf{Condition (iv):}} By orthogonality
\begin{align*}
  \sum_{q=Q+1}^\infty q!\|g_{N,q}\|^2 = \mathcal{H}^d(B^d_N)^{-1}\E\left[ \left(\sum_{q=Q+1}^\infty \sum_{|n|=q}\int_{G(d,d-m)}c(n)\int_{B^d_N}\widetilde H_n(Y^F(t))\,dt \,\nu(dF)\right)^2\right].
\end{align*}
The same computations as in the verification of condition (ii) bound the latter by
\begin{align*}
  &\frac{2^dK}{\rho} \frac{\#\{\Z^d\cap B^d_{N+d}\}}{\mathcal{H}^d(B^d_N)}\int_{\R^d}\psi(t)dt \sum_{q=Q+1}^\infty q^D(c\rho)^q + \frac{2(2s+1)^d}{\mathcal{H}^d(B^d_N)}\\
  &\quad\times\sum_{z\in \Z^d\cap B^d_{N+d}}\E\left[ \left(\sum_{q=Q+1}^\infty \sum_{|n|=q}\int_{G(d,d-m)}c(n)\int_{[0,1)^d}\ind\{t+z\in B^d_N\}\widetilde H_n(Y^F(t))\,dt \,\nu(dF)\right)^2\right].
\end{align*}
In the limit $N\to\infty$ and then $Q\to \infty$ the first summand vanishes, since the series is the tail of a convergent series. The second summand needs more attention. We first split the summation into one over the indices $I_1^N:=\{z\in \Z^d\cap B^d_{N+d}\mid z+[0,1)^d\subset B^d_N\}$ and one over $I_2^N:=\{z\in \Z^d\cap B^d_{N+d}\mid z+[0,1)^d\cap (B^d_N)^\mathsf{c}\neq \emptyset\}$. The sum over $I_1^N$ can be bounded by
\begin{align*}
  2(2s+1)^d\E\left[ \left(\sum_{q=Q+1}^\infty \sum_{|n|=q}\int_{G(d,d-m)}c(n)\int_{[0,1)^d}\widetilde H_n(Y^F(t))\,dt \,\nu(dF)\right)^2\right]
\end{align*}
since $\#I_1^N\mathcal{H}^d(B^d_N)^{-1}\leq 1$. Hence the latter tends to zero for $Q\to \infty$ by Lemma \ref{thm:L^2Approx}. We bound the summation over $I_2^N$ by
\begin{align*}
  \frac{2(2s+1)^d\#I_2^N}{\mathcal{H}^d(B^d_N)} \sup_{W \subset [0,1)^d\mathrm{\, convex}}\E\left[ \left(\sum_{q=1}^\infty \sum_{|n|=q}\int_{G(d,d-m)}c(n)\int_{W}\widetilde H_n(Y^F(t))\,dt \,\nu(dF)\right)^2\right].
\end{align*}
Lemma \ref{lem:auxBounds}\ref{lem:uniBound} yields the upper bound
\begin{align*}
  c2(2s+1)^d \#I_2^N \mathcal{H}^d(B^d_N)^{-1},
\end{align*}
which vanishes for $N\to \infty$. This shows the assertion.

\section{A lower bound for the asymptotic variance}

We follow the lines of \cite [Lemma 2.2]{paper:EstradeLeon} and give a lower bound for the asymptotic variance.
\begin{lemma}\label{lem:lowerBound}
  Let $X$ be a real Gaussian field on $\R^d$, which satisfies the assumptions (A1)--(A3). We then have for $m\in \{0,\ldots,d-1\}$ and $\sigma_m^2$ given as in Theorem \ref{thm:mainTheorem} that
  \begin{align*}
    \sigma^2_m \geq {d \brack d-m }^2(2\pi)^mf(0)H_{d-m}(u)^2\phi(u)^2.
  \end{align*}
\end{lemma}

\begin{proof}
  Recall that according to Theorem \ref{thm:NourdinPeccati} the asymptotic variance is given by $\sum_{q\geq 1}\sigma^2_{m,q}$, where $\sigma_{m,q}^2$ is defined as the limit in condition (i) of that theorem. Hence, we obtain a lower bound for the asymptotic variance by computing $\sigma_{m,1}^2$. By (\ref{eq:asymptoticVariance})
  \begin{align*}
    \sigma_{m,1}^2=\sum_{k,l\in\{1,\ldots, D\}}b(k)b(l)\int_{\R^d}\int_{G(d,d-m)}\int_{G(d,d-m)} \E\left[Y^F_{k}(t)Y^{F'}_{l}(0)\right]\nu(dF)\nu(dF')\,dt,
  \end{align*}
  where the coefficients $b(\cdot)$ are given by
  \begin{align*}
    b(k)=\sum_{n \in \N^D,|n|=1} \ind\{k\in \A_{n}\} \frac{c(n)}{|\A_{n}|}.
  \end{align*}
  The sets $\A_{n}$ consist of only one element, namely the number of the component of $n$, which contains the $1$. Thus if we write $e_{i}\in \R^{D}$ for the vector, whose components are $0$ except for the $i$-th component, which is $1$, we obtain
  \begin{align*}
    b(k)=c(e_k).
  \end{align*}
  By the definition of the coefficients $c(\cdot)$, cf. (\ref{def:coefc}), we see that $c(e_k)=0$ for $k=1,\ldots,d-m$ and therefore obtain
  \begin{align*}
    \sigma_{m,1}^2=\sum_{k,l=d-m}^D c(e_k)c(e_l)\int_{\R^d}\int_{G(d,d-m)}\int_{G(d,d-m)} \E\left[Y^F_{k}(t)Y^{F'}_{l}(0)\right]\,\nu(dF)\,\nu(dF')\,dt.
  \end{align*}
  We now show that
  \begin{align*}
    \int_{\R^d}\E\left[\mathcal{X}^F_{k}(t)\mathcal{X}^{F'}_{l}(0)\right] \,dt = (2\pi)^d f(0)\delta_{k,l}(D,D),
  \end{align*}
  for $F,F'\in A(d,d-m)$ and $k,l=1,\ldots, D$.
  Consider the case $(k,l)=(D,D)$. Then the equality
  $\E\left[\mathcal{X}^F_{D}(t)\mathcal{X}^{F'}_{D}(0)\right] = \E\left[X(t)X(0)\right] = (2\pi)^{d/2}\mathcal{F}(f)(t)$ holds,
  where $\mathcal{F}$ denotes the Fourier transformation. By \ref{as:covariance} the spectral density $f$ is continuous and $\E\left[X(t)X(0)\right]$ is integrable, which yields that $\int_{\R^d}\E\left[X(t)X(0)\right] dt=(2\pi)^df(0)$, via the Fourier cotransformation. In the cases, where $(k,l)\neq (D,D)$, at least one of the factors $\mathcal{X}^F_{D}$ or $\mathcal{X}^{F'}_{D}$ is at least one directional derivative of the field $X$, say in direction $u\in S^{d-1}$. This yields that $\E\left[\mathcal{X}^F_{k}(t)\mathcal{X}^{F'}_{l}(0)\right]$ equals, up to a power of $-1$, the function $\frac{\partial}{\partial u} g$, where $g$ is either the covariance function or a derivative of it. Thus by Fubini's theorem, we conclude that
  \begin{align*}
    \int_{\R^d}\E\left[\mathcal{X}^F_{k}(t)\mathcal{X}^{F'}_{l}(0)\right] dt &= \int_{\R}\ldots \int_{\R} \frac{\partial}{\partial u} g(t_1,\ldots,t_d) \,dt_1 \ldots dt_d \\
    &=\sum_{i=1}^d u^{(i)} \int_{\R}\ldots \int_{\R} \frac{\partial}{\partial t_i} g(t_1,\ldots,t_d) \,dt_i\, dt_{1} \ldots\overline{dt_i}\ldots dt_d \\
    &= \sum_{i=1}^d u^{(i)} \int_{\R}\ldots \int_{\R} \left .g(t_1,\ldots,t_d)\right |_{t_i=-\infty}^\infty \,dt_{1} \ldots\overline{dt_i}\ldots dt_d\\
    &=0,
  \end{align*}
  where we used in the last line assumption \ref{as:covariance}, that is, the covariance function and its derivatives tend to 0 for $\|t\|\to \infty$.

  The definition of $Y$, cf. \eqref{def:Y}, implies
  \begin{align*}
    \E\left[Y^F_{k}(t)Y^{F'}_{l}(0)\right]= \sum_{r=1}^D\sum_{s=1}^D\Lambda_{l,r}^{-1}\Lambda_{k,s}^{-1} \E\left[\mathcal{X}^F_{r}(t)\mathcal{X}^{F'}_{s}(0)\right],
  \end{align*}
  which yields
  \begin{align*}
    \int_{\R^d} \E\left[Y^F_{k}(t)Y^{F'}_{l}(0)\right]\, dt = \Lambda_{l,D}^{-1}\Lambda_{k,D}^{-1} (2\pi)^df(0)
  \end{align*}
  and we conclude with Fubini's theorem, that
  \begin{align*}
    \sigma_{m,1}^2={d \brack d-m }^2\sum_{k,l=d-m}^D c(e_k)c(e_l)\Lambda_{l,D}^{-1}\Lambda_{k,D}^{-1} (2\pi)^df(0) = {d \brack d-m }^2c(e_D)^2(\Lambda^{-1}_{D,D})^2(2\pi)^df(0),
  \end{align*}
  where the last equality holds since $\Lambda$ is lower triangular.
  In order to calculate the coefficients $c(e_D)$, we have to analyse the covariance structure of $\mathcal{X}^F$. We first write the $K+1:=(d-m)(d-m+1)/2+1$ last coordinates of $\mathcal{X}^F$ in the order
  \begin{align*}
    \left( \left (\frac{\partial^2}{\partial v_i\partial v_j} X \right)_{1\leq i < j\leq d-m}, \left(\frac{\partial^2}{\partial v_i\partial v_i} X\right)_{i=1}^{d-m} , X   \right).
  \end{align*}
  Thus, using stationarity, isotropy and $\Co^X(0)=1$, the covariance matrix of this vector at $0$ is given by
  \begin{align*}
    \begin{pmatrix}
      \Co\left(\left(\frac{\partial^2}{\partial t_i\partial t_j} X\right )_{ i < j}\right) & \Co\left(\left (\frac{\partial^2}{\partial t_i\partial t_j} X \right)_{i< j}, \left(\frac{\partial^2}{\partial t_i\partial t_i} X\right)_{i=1}^{d-m}\right) & 0 \\
      \Co\left(\left(\frac{\partial^2}{\partial t_i\partial t_i} X\right)_{i=1}^{d-m},\left (\frac{\partial^2}{\partial t_i\partial t_j} X \right)_{i < j} \right)& \Co\left(\left(\frac{\partial^2}{\partial t_i\partial t_i} X\right)_{i=1}^{d-m}\right) & -1 \\
      0 & -1 & 1 \\
    \end{pmatrix},
  \end{align*}
  which equals the product $\Lambda_2\Lambda_2^\top$, where $\Lambda_2\in\R^{K+1}$ is a lower triangular matrix, given in Lemma \ref{lem:decom}. We choose the matrix $L\in \R^{K\times K}$, the vector $l\in \R^{K}$ and $\alpha>0$ such that
  \begin{align*}
    \Lambda_2=\begin{pmatrix}
                L & 0 \\
                l^\top& \alpha \\
              \end{pmatrix}.
  \end{align*}
  Then the relation $\|l\|^2+\alpha^2 =1$ holds as well as $Ll=( 0_{1\times K}, -1_{1\times d-m}).$ With this specific representation of $\Lambda_2$, we have
  \begin{align*}
    c(e_D)&=(2\pi)^{-(d-m)/2}(-1)^{d-m}\int_{\R^{K}\times \R} \det(Ly)\ind\{\langle l,y\rangle +\alpha z \geq u\}z\phi_{K}(y)\phi(z) \,d(y,z)\\
    &=-(2\pi)^{-(d-m)/2}(-1)^{d-m}\int_{\R^{K}\times \R} \det(Ly)\ind\{\langle l,y\rangle +\alpha z \geq u\}\phi_{K}(y)\phi'(z) \,d(y,z)\\
    &=(2\pi)^{-(d-m)/2}(-1)^{d-m}\int_{\R^{K}} \det(Ly)\phi_{K}(y)\phi\left(\alpha^{-1}(u-\langle l,y\rangle)\right) \,dy,
  \end{align*}
  where we used that $z\phi(z)=-\phi'(z)$ in the first line and Fubini's theorem in the second. Using the Hermite expansion of $y\mapsto \det(Ly)$ given in \cite [Lemma A.2]{paper:EstradeLeon}, we obtain
  \begin{align*}
    c(e_D)&=(2\pi)^{-(d-m)/2}(-1)^{d-m} \sum_{\substack{m\in \N^{K},|m|=d-m}}\beta_m\int_{\R^K}\widetilde{H}_m(y)\phi_{K}(y)\phi\left(\alpha^{-1}(u-\langle l,y\rangle)\right) \,dy\\
    &=(2\pi)^{-(d-m)/2} \sum_{\substack{m\in \N^{K},|m|=d-m}}\beta_m \int_{\R^K}D^m\phi_{K}(y)\phi\left(\alpha^{-1}(u-\langle l,y\rangle)\right) \,dy,
  \end{align*}
  where $D^m\phi$ denotes $\frac{\partial^{|m|}}{\partial t_{m_1}\ldots \partial t_{m_{K}}}\phi$ and $\beta_m$ are real coefficients. Following the argument in \cite{paper:EstradeLeon}, we define $h\colon \R^K\to \R, x\mapsto \phi(\alpha^{-1}\langle l,y\rangle)$ and choose $l^\prime$ such that $\langle l , l^\prime\rangle=1$. We then obtain
  \begin{align*}
    \int_{\R^K}D^m\phi_{K}(y)\phi\left(\alpha^{-1}(u-\langle l,y\rangle)\right) \,dy = (h\ast D^m\phi_K)(ul^\prime)=D^m(h \ast \phi_K)(ul^\prime).
  \end{align*}
  By \cite [Remark A.4]{paper:EstradeLeon}, which reads $(h \ast \phi_K)(y)= \alpha\phi(\langle l,y\rangle)$ for $y\in \R^K$, we obtain
  \begin{align*}
    D^m(h \ast \phi_K)(y)=\alpha l^{(m)} \phi^{(d-m)}(\langle l , y\rangle)=(-1)^{d-m}\alpha l^{(m)} H_{d-m}(\langle l,y\rangle)\phi(\langle l,y\rangle).
  \end{align*}
  Thus by \cite [Lemma A.2]{paper:EstradeLeon} in the second equality
  \begin{align*}
    c(e_D)&= (2\pi)^{-(d-m)/2} \sum_{\substack{m\in \N^{K}\\|m|=d-m}}\beta_m l^{(m)} (-1)^{d-m}\alpha H_{d-m}(u)\phi(u) \\
    &= (2\pi)^{-(d-m)/2}\det(Ll)(-1)^{d-m}\alpha H_{d-m}(u)\phi(u).
  \end{align*}
  Note that the $K$-dimensional vector $Ll$ corresponds to the symmetric $(d-m)\times (d-m)$\nobreakdash-\hspace{0pt}matrix, whose nondiagonal entries are given by the first $(d-m)(d-m-1)/2$ entries of $Ll$ and whose diagonal is given by the $d-m$ last entries of $Ll$, thus $\det(Ll)=(-1)^{d-m}$. Hence we obtain
  \begin{align*}
    c(e_D)=(2\pi)^{-(d-m)/2}\alpha H_{d-m}(u)\phi(u)
  \end{align*}
  and therefore conclude as asserted
  \begin{equation*}
    \sigma_{m,1}^2={d \brack d-m }^2(2\pi)^{m}f(0)H_{d-m}(u)^2\phi(u)^2.\qedhere
  \end{equation*}
\end{proof}

\appendix
\section{Proofs of statements holding almost surely}
\begin{lemma}\label{lem:pre}
  Let $X\colon \Omega\times \R^d\to \R$ be an almost surely of class $C^2$, stationary Gaussian field satisfying \ref{as:nondegenerate}. Then for almost all $\omega\in\Omega$ there is a measurable set $A'(\omega)\subset A(d,d-m)$, where $\mu(A'(\omega)^\mathsf{c})=0$, such that
  \begin{align*}
    &\mathbb{P}\left(\exists F\in A' \exists t\in \overline{B^d_N}\cap F: \nabla(X|_F)(t)=0,
    X(t)=u\right )=0,
    \\&\mathbb{P}\left(\exists F\in A' \exists t\in \overline{B^d_N}\cap F: \nabla(X|_F)(t)=0,\det(D^2(X|_F)(t))=0\right )=0,\\
    &\mathbb{P}\left(\exists F\in A' \exists t\in \partial\overline{B^d_N}\cap F: \nabla(X|_F)(t)=0\right )=0.
  \end{align*}
\end{lemma}

\begin{proof}
  We show the details for the second equality. By an application of \cite [Lemma 11.2.11]{book:AdlerTaylor}, choose $T:=\overline{ B^{d-m}_{c(F,N)}}$, where $c(F,N)$ denotes the radius of $B^d_N\cap F$, $f:=\nabla X^F_{v(F)}$, $u:=0$, cf. (\ref{def:XRotiert}), we obtain
  \begin{align*}
    \mathbb{P}(\exists t\in \overline{ B^{d-m}_{c(F,N)}}: \nabla X^F_{v(F)}(t)=0,\det(D^2X^F_{v(F)}(t))=0) =0
  \end{align*}
  which yields by Fubini
  \begin{align*}
    \E&\left[\int_{A(d,d-m)}\ind \{\exists t\in \overline{B^d_N}\cap F: \nabla(X|_F)(t)=0,\det(D^2(X|_F)(t))=0\} \,\mu(dF) \right]\\
    &=\int_{A(d,d-m)}\mathbb{P}(\exists t\in \overline{B^d_N}\cap F: \nabla(X|_F)(t)=0,\det(D^2(X|_F)(t))=0) \,\mu(dF)\\
    &=\int_{A(d,d-m)}\mathbb{P}(\exists t\in \overline{ B^{d-m}_{c(F,N)}}: \nabla X^F_{v(F)}(t)=0,\det(D^2X^F_{v(F)}(t))=0) \,\mu(dF)\\
    &=0.
  \end{align*}
  In order to show that the above integrand is $\mathbb{P}\otimes \mu$ measurable, we define for $t\in \R^d$ the function $f_t\colon \Omega\times A(d,d-m) \to \R^{d+1}$ by
  \begin{align*}
    f_t(\omega,F):=\left (\pi_{F^\circ}(\nabla (X(\omega))(t)),\det(\pi_{F^\circ}\circ D^2 (X(\omega))(t)\circ \pi_{F^\circ})\right ),
  \end{align*}
  where $\pi_{F^\circ}$ denotes the projection onto $F^\circ$. Note that $f_t(\cdot,F)$, $F\in A(d,d-m)$, is measurable and $f_t(\omega,\cdot)$, $\omega \in \Omega$, is continuous, which implies that $f_t$ is measurable. Moreover the function $f_\cdot(\omega,F)$ is continuous on $\R^d$ and for $t\in F$ the equalities $\pi_{F^\circ}(\nabla (X(\omega))(t))=\nabla(X(\omega)|_F)(t)$ and $\pi_{F^\circ}\circ D^2 (X(\omega))(t)\circ \pi_{F^\circ}=D^2(X(\omega)|_F)(t)$ hold.
  We deduce the measurability of the integrand from the following
  \begin{align*}
    &\{(\omega,F)\in \Omega\times A(d,d-m)\colon\exists t\in \overline{B^d_N}\cap F: \nabla(X(\omega)|_F)(t)=0,\det(D^2(X(\omega)|_F)(t))=0\}\\
    &=\bigcap _{n\in \N} \bigcup_{t\in  I} \left \{ (\omega,F)\in \Omega\times A(d,d-m)\colon f_t(\omega,F)\in B^{d+1}_{\frac{1}{n}}, B^d_{\frac{1}{n}}(t)\cap F\neq \emptyset \right\}\\
    &=\bigcap _{n\in \N} \bigcup_{t\in I} f_t^{-1}\left ( B^{d+1}_{\frac{1}{n}}\right ) \cap \left ( \Omega \times (\mathcal{F}_{B^d_{\frac{1}{n}}(t)}^d\cap A(d,d-m))\right),
  \end{align*}
  where $I$ denotes a dense subset of $\overline{B^d_{N}}$ and $\mathcal{F}_{B^d_{\frac{1}{n}}(t)}^d\cap A(d,d-m)$ is open in $A(d,d-m)$, cf. the discussion after \cite [Theorem 13.2.5]{book:SchneiderWeil}.

  Thus we obtain the existence of the set $B_2\in \mathcal{F}\otimes \mathcal{B}(A(d,d-m))$, such that $\mathbb{P}\otimes \mu(B_2^\mathsf{c})=0$ and for all $(\omega,F)\in B_2$, we have that
  \begin{align*}
    \ind&\{\exists t\in \overline{B^d_N}\cap F: \nabla(X(\omega)|_F)(t)=0 ,\det(D^2(X(\omega)|_F)(t))=0\}=0\\
    &\Leftrightarrow \forall t\in \overline{B^d_N}\cap F: \lnot(\nabla(X(\omega)|_F)(t)=0 ,\det(D^2(X(\omega)|_F)(t))=0).
  \end{align*}
  We now define for $\omega\in \Omega$ the $\omega$-cross section of $B_2$ as
  \begin{align*}
    B_{2,\omega}:=\{F\in A(d,d-m)\mid (\omega,F)\in B\}
  \end{align*}
  and observe, cf. \cite [Thm. 1.22]{book:EvansGariepy}, that for almost all $\omega \in \Omega$
  \begin{align*}
    \mu(B_{2,\omega}^\mathsf{c})=0.
  \end{align*}
  Similar reasoning, except that we use \cite [Lemma 11.2.10]{book:AdlerTaylor} and \cite [Lemma 11.2.12]{book:AdlerTaylor}, yields sets $B_1,B_3 \in \mathcal{F}\otimes \mathcal{B}(A(d,d-m))$ and cross sections $B_{1,\omega}, B_{3,\omega}$, whose complements have $\mu$ measure zero for almost all $\omega \in \Omega$. Thus for almost all $\omega$ the complement of $A'(\omega):=\cap_{i=1}^3B_{i,\omega}$ has $\mu$ measure zero, and we conclude
  \begin{align*}
    \mathbb{P}&(\exists F \in A' \exists t\in \overline{B^d_N}\cap F:\nabla(X(\omega)|_F)(t)=0 ,\det(D^2(X(\omega)|_F)(t))=0 )\\
    &=\mathbb{P}(\omega \in \Omega \mid \exists F \in A(d,d-m): (w,F)\in \cap_{i=1}^3B_{i} \\ &\hspace{2.4cm}\exists t\in \overline{B^d_N}\cap F:\nabla(X(\omega)|_F)(t)=0 ,\det(D^2(X(\omega)|_F)(t))=0)\\
    &=\mathbb{P}(\emptyset)=0.
  \end{align*}
  And analogously
  \begin{align*}
    &\mathbb{P}(\exists F\in A' \exists t\in \overline{B^d_N}\cap F: \nabla(X|_F)(t)=0, X(t)=u)=0,\\
    &\mathbb{P}(\exists F\in A' \exists t\in \partial \overline{B^d_N}\cap F: \nabla(X|_F)(t)=0)=0.\qedhere
  \end{align*}
\end{proof}

\begin{lemma} \label{lem:trafo}
  Let $F\in A(d,d-m)$ and $W\subset \R^d$ be convex and bounded. Moreover let $X\colon \Omega\times \R^d\to \R$ be a Gaussian field satisfying the assumptions:
  \begin{enumerate}[label={\normalfont(\roman*)}]
    \item $X$ has almost surely $\mathcal{C}^2$ paths.
    \item There are almost surely no points $t\in \overline{W}\cap F$
    \begin{enumerate}[label={\normalfont(\alph*)}]
      \item such that $\nabla (X|_F)(t)=0$ and $X(t)=u$.
      \item such that $\nabla (X|_F)(t)=0$ and $\det(D^2 (X|_F))=0$.
    \end{enumerate}
    \item There are almost surely no points $t\in \partial W \cap F$ with $\nabla( X|_F)(t)=0$.
  \end{enumerate}
  Then
  \begin{align*}
    &\#\{t\in W\cap F:X(t)\geq u, \nabla (X|_F)(t)=0,\iota_{-X,W\cap F}(t) \mathrm{\, even}\}\\
  &- \#\{t\in W\cap F:X(t)\geq u, \nabla (X|_F)(t)=0,\iota_{-X,W\cap F}(t) \mathrm{\, odd}\}
  \\ &\stackrel{a.s.}{=} (-1)^{d-m}\lim_{\varepsilon \to 0} \int_{W\cap F} \delta_{\varepsilon}(\nabla(X|_F)(t))\ind\{X(t)\geq u\} \det(D^2(X|_F)(t))\, dt.
  \end{align*}
\end{lemma}

\begin{proof}
  We follow the proof of \cite [Theorem 11.2.3]{book:AdlerTaylor}.

  We consider the points $t_1,\ldots, t_n \in \overline{W}\cap F$ with $\nabla(X|_F)(t_i)=0$ and $\iota_{-X,W\cap F}(t_i) \text{ even}$, for $i=1,\ldots, n$ and note that there are only finitely many because of (ii)(b), the fact that $\overline{W}\cap F$ is compact and the implicit function  theorem. Moreover, condition (iii) implies the existence of open sets, with respect to the subspace topology, $U_i \subset \overline{W}\cap F$ such that $t_i \in U_i$, the sets $U_1,\ldots, U_n$ are pairwise disjoint and $U_i\cap \partial W \cap F= \emptyset $, for $i=1,\ldots, n.$

  Condition (ii)(a) yields that we can choose the open sets $U_i$, $i=1,\ldots, n$, small enough such that either for all $t\in U_i$ we have $X(t)\in (u,\infty)$ or for all $t\in U_i$ we have $X(t)\in (-\infty,u)$.
  The same line of reasoning yields open sets $U'_1,\ldots, U'_{n'}\subset \overline{W} \cap F$ containing the points $t'_1,\ldots, t'_{n'}$ with $\nabla(X|_F)(t'_i)=0$ and $\iota_{-X,W\cap F}(t'_i) \text{ odd}$, for $i=1,\ldots, n'$, satisfying the same properties as $U_1,\ldots, U_n$.
  By continuity of the determinant and condition (ii)(b), we can choose the sets $U_1,\ldots,U_n, U'_1,\ldots, U'_{n'}$ small enough such that the sign of $\det(D^2(X|_F))$ stays constant on those sets. The last observation we need is that by contradiction, cf. \cite [Lemma 11.2.3]{book:AdlerTaylor}, we see the existence of a number $\varepsilon>0$ small enough such that
  \begin{align}\label{eq:trafo1}
    \nabla(X|_F)^{-1}(B^d_\varepsilon) \cap \overline{W}\cap F \subset \bigcup_{i=1}^n U_i\cup \bigcup_{i=1}^{n'} U'_{i}.
  \end{align}

  Now, by the inverse function theorem, we can choose the sets $U_1,\ldots,U_n, U'_1,\ldots, U'_{n'}$ and the number $\varepsilon$ small enough to obtain $\nabla(X|_F)$ bijective on the sets $U_i$ and onto $B^{d-m}_\varepsilon\subset F^\circ$. Notice the abuse of notation in writing $\nabla(X|_F)^{-1}$ for every inverse. Hence we have
  \begin{align*}
    &\#\{t\in W\cap F:X(t)\geq u, \nabla (X|_F)(t)=0,\iota_{-X,W\cap F}(t) \text{ even}\}\\
  &- \#\{t\in W\cap F:X(t)\geq u, \nabla (X|_F)(t)=0,\iota_{-X,W\cap F}(t) \text{ odd}\}\\
  &=\sum_{i=1}^n \int_{\nabla (X|_F)(U_i)} \delta_{\varepsilon}(y) \ind\{X(\nabla(X|_F)^{-1}(y))\geq u\} \,\mathcal{H}^{d-m}(dy)\\
  &\quad - \sum_{i=1}^{n'} \int_{\nabla (X|_F)(U'_i)} \delta_{\varepsilon}(y) \ind\{X(\nabla(X|_F)^{-1}(y))\geq u\} \,\mathcal{H}^{d-m}(dy).
  \end{align*}
  We obtain with the substitution rule the equality to
  \begin{align}\label{eq:trafo}
    \sum_{i=1}^n&\int_{U_i} \delta_\varepsilon(\nabla(X|_F)(t))\ind\{X(t)\geq u\} |\det(D^2(X|_F)(t))| \,\mathcal{H}^{d-m}(dt)\notag\\
    & -\sum_{i=1}^{n'}\int_{U'_i} \delta_\varepsilon(\nabla(X|_F)(t))\ind\{X(t)\geq u\} |\det(D^2(X|_F)(t))| \,\mathcal{H}^{d-m}(dt).
  \end{align}
  Since the sign of $\det(D^2(X|_F))$ is constant on the sets $U_1,\ldots,U_n, U'_1,\ldots, U'_{n'}$ and furthermore, by definition of $\iota$, the equality $\operatorname{sign}(\det(D^2(X|_F(t_i))))=(-1)^{d-m-\iota_{-X,W\cap F}(t_i)}$ holds as well as the same relation for the points $t'_i$, we have
  \begin{align*}
    &\operatorname{sign}(\det(D^2(X|_F(t)))) = (-1)^{d-m}, \quad \text{for all } t \in U_i,\\
      &\operatorname{sign}(\det(D^2(X|_F(t)))) = -(-1)^{d-m}, \quad \text{for all } t \in U'_j,
  \end{align*}
  for $i=1,\ldots, n$ and $j=1,\ldots, n'$. Therefore (\ref{eq:trafo}) equals
  \begin{align*}
    (-1)^{d-m}\bigg(\sum_{i=1}^n&\int_{U_i} \delta_\varepsilon(\nabla(X|_F)(t))\ind\{X(t)\geq u\} \det(D^2(X|_F)(t))\, \mathcal{H}^{d-m}(dt)\notag\\
    & +\sum_{i=1}^{n'}\int_{U'_i} \delta_\varepsilon(\nabla(X|_F)(t))\ind\{X(t)\geq u\} \det(D^2(X|_F)(t)) \,\mathcal{H}^{d-m}(dt)\bigg),
  \end{align*}
  which yields together with (\ref{eq:trafo1}) the assertion. \qedhere
\end{proof}

\section{Proof of Lemma \ref{lem:hilfL^2}}
In the remaining part of the appendix we give a proof of Lemma \ref{lem:hilfL^2}. Lemmata \ref{lem:detformula} -- \ref{lem:boundedExp2} are invoked in the proof of Lemma \ref{lem:hilfL^2} (i).
\begin{proof}[Proof (Lemma \ref{lem:hilfL^2})]
  To prove (i) we use refined methods of \cite [Proposition 1.1 (1)]{paper:EstradeLeon}. In order to apply the Rice formulas, cf. \cite [Chapter 6]{book:AzaisWschebor}, we define the Gaussian field $X^F_{v(F)}$ on $\R^{d-m}$, for $F\in A(d,d-m)$ and with $v(F):=(v_1,\ldots,v_{d-m})$ denoting an orthonormal basis of $F^\circ$, by setting
  \begin{align} \label{def:XRotiert}
    X^F_{v(F)}(s):= (X\circ \rho_{v(F)}^F)(s),
  \end{align}
  where $\rho_{v(F)}^F\colon \R^{d-m}\to F\subset\R^d$ is defined as $x\mapsto \sigma^{v(F)}(x)+ p$, $\sigma^{v(F)}\colon \R^{d-m}\to \R^d$ is given by $x\mapsto (v_1|\ldots|v_{d-m})x$ and $p \in F$ is such that $d(0,F)=\inf \{|y|\mid y\in F\}=d(0,p)$. We then have
  \begin{align*}
    \nabla X^F_{v(F)}(t) = \nabla_{v(F)} (X)(\rho_{v(F)}^F (t)) \quad\text{ and }\quad D^2 X^F_{v(F)}(t)= D_{v(F)}^2 (X)(\rho_{v(F)}^F(t))
  \end{align*}
  for $t\in \R^{d-m}$. Since $\nabla (X|_F)(t)= \sum_{i=1}^{d-m}\partial/\partial v_i X(t) v_i$, we obtain for $y\in F^\circ$ that
  \begin{align*}
    \#\{t\in W\cap F:\nabla (X|_F)(t)=y\}= \#\{s\in W^F_{v(F)}:\nabla X^F_{v(F)}(s)=(\langle y,v_1\rangle, \ldots,\langle y, v_{d-m}\rangle)\},
  \end{align*}
  where $W^F_{v(F)}$ denotes $(\rho^F_{v(F)})^{-1}(W\cap F)$. Note that $\on{diam}(W^F_{v(F)})\leq \on{diam}(W)<\infty$ and that we abbreviate $y^{v(F)}:=(\sigma^{v(F)})^{-1}(y)=(\langle y,v_1\rangle,\ldots,\langle y, v_{d-m}\rangle)$.

  By \ref{as:dif} and \cite [Proposition 6.5]{book:AzaisWschebor} the assumptions of the Rice formula in \cite [Theorem 6.2]{book:AzaisWschebor} are satisfied for fixed $F$ and we therefore obtain
  \begin{align*}
    \E&\left[ \#\{t\in W^F_{v(F)}:\nabla X^F_{v(F)}(t)=y^{v(F)}\}\right] \\
    &= \int_{W^F_{v(F)}} \E\left[ |\det D^2 X^F_{v(F)}(t)| \mid \nabla X^F_{v(F)}(t)=y^{v(F)}\right] p_{\nabla X^F_{v(F)}(t)}(y^{v(F)}) \,dt,
  \end{align*}
  where $p_{\nabla X^F_{v(F)}(t)}(y^{v(F)})$ denotes the probability density of $\nabla X^F_{v(F)}(t)$ at point $y^{v(F)}$. Stationarity and isotropy imply that $\left( \frac{\partial}{\partial v_i} X(t)\right)_{i=1}^{d-m} \stackrel{\mathcal{D}}{=}\left (\frac{\partial}{\partial t_i} X(0)\right)_{i=1}^{d-m}$ and that the first and second derivatives are independent at equal times. Thus the above equals
  \begin{align}\label{eq:rice}
    \int_{W^F_{v(F)}}& \E\left[ |\det D_{v(F)}^2 (X)(\rho_{v(F)}^F(t))| \mid \nabla_{v(F)} (X)(\rho_{v(F)}^F (t))=y^{v(F)}\right] p_{\nabla_{v(F)} (X)(\rho_{v(F)}^F(t))}(y^{v(F)}) \,dt\notag\\
    &=\int_{W^F_{v(F)}} \E\left[ |\det D_{v(F)}^2 X(0)|\right] p_{\nabla_{v(F)} X(0)}(y^{v(F)}) \,dt\notag\\
    &= \E\left[ |\det D_{v(F)}^2 X(0)|\right] p_{\frac{\partial}{\partial t_1}X(0)\ldots\frac{\partial}{\partial t_{d-m}}X(0)}(y^{v(F)}) \mathcal{H}^{d-m}(W \cap F)\\
    &\leq\E\left[ |\det D_{v(F)}^2 X(0)|\right] p_{\frac{\partial}{\partial t_1}X(0)\ldots\frac{\partial}{\partial t_{d-m}}X(0)}(0) \mathcal{H}^{d-m}(W \cap F) .\notag
  \end{align}
  Observe that
  \begin{align*}
    \E\left[ |\det D_{v(F)}^2 X(0)|\right]\leq \E\left[1+ \det (D_{v(F)}^2 X(0))^2\right]
  \end{align*}
  and that by Hadamard's inequality, cf. \cite [Fact 8.17.11]{book:Bernstein}
  \begin{align*}
    \det (D_{v(F)}^2 X(0))^2\leq \prod_{i=1}^{d-m}\sum_{k=1}^{d-m}\left(\frac{\partial^2}{\partial v_i \partial v_k}X(0)\right)^2 = \sum_{k_1=1}^{d-m}\ldots \sum_{k_{d-m}=1}^{d-m}\prod_{i=1}^{d-m}\left (\frac{\partial^2}{\partial v_i \partial v_{k_i}}X(0)\right)^2.
  \end{align*}
  Hence, we obtain with the definition $Y_j^{k}:=\frac{\partial^2}{\partial v_r \partial v_{k_{\lfloor (j+1)/2\rfloor}}}X(0)$ for $j=1,\ldots, 2(d-m)$, that
  \begin{align*}
    \E\left[\det (D_{v(F)}^2 X(0))^2\right] &\leq \sum_{k_1=1}^{d-m}\ldots \sum_{k_{d-m}=1}^{d-m}\E\left[ \prod_{j=1}^{2(d-m)}Y_j^k\right]\\
    &=\sum_{k_1=1}^{d-m}\ldots \sum_{k_{d-m}=1}^{d-m} \sum \E\left[ Y_{j_1}^kY_{j_2}^k\right]\ldots \E\left[ Y_{j_{2(d-m)-1}}^kY_{j_{2(d-m)}}^k\right],
  \end{align*}
  where the last line follows from Wick's formula, cf. \cite [Lemma 11.6.1]{book:AdlerTaylor}, and the sum is taken over the $(2(d-m))!/(2^{d-m}(d-m)!)$ possibilities of choosing $d-m$ pairs of $Y_1^k,\ldots ,Y_{2(d-m)}^k$, where the order of the pairs does not matter. We conclude from $\E\left[ Y_{j}^kY_{j'}^k\right]\leq \widetilde \psi(0)\leq d^2\psi(0)$, cf. \ref{as:covariance}, that
  \begin{align*}
    \E\left[\det (D_{v(F)}^2 X(0))^2\right] \leq c \psi^{d-m}(0)
  \end{align*}
  and that the expectation is finite independently of $F$.

  By \ref{as:dif}, \ref{as:nondegenerate} and \cite [Proposition 6.5]{book:AzaisWschebor} the required conditions of the Rice formula in \cite [Theorem 6.3]{book:AzaisWschebor} are satisfied and an application of the latter yields
  \begin{align*}
    &\E\left[ \#\{t\in W^F_{v(F)}:\nabla X^F_{v(F)}(t)=y^{v(F)}\}(\#\{t\in W^F_{v(F)}:\nabla X^F_{v(F)}(t)=y^{v(F)}\}-1)\right]\\
    &=\int_{W^F_{v(F)}}\int_{W^F_{v(F)}} \E\left[ |\det D^2X^F_{v(F)}(t)\det D^2X^F_{v(F)}(t')| \mid \nabla X^F_{v(F)}(t)=\nabla X^F_{v(F)}(t')=y^{v(F)} \right]\notag\\
    &\hspace{2.3cm}\times p_{\nabla X^F_{v(F)}(t)\nabla X^F_{v(F)}(t')}(y^{v(F)},y^{v(F)}) \,dt \,dt',
  \end{align*}
  where $p_{\nabla X^F_{v(F)}(t_1)\nabla X^F_{v(F)}(t_2)}(y^{v(F)},y^{v(F)})$ denotes the density of the $2(d-m)$-dimensional random vector $(\nabla X^F_{v(F)}(t),\nabla X^F_{v(F)}(t'))$ at point $(y^{v(F)},y^{v(F)})$. By stationarity  and Fubini's theorem the above equals
  \begin{align}\label{eq:ricefac}
    &\int_{W^F_{v(F)}-W^F_{v(F)}}\E\left[ |\det D^2_{v(F)}(X)(\rho_{v(F)}^F(t))\det D^2_{v(F)}(X)(\rho_{v(F)}^F(0))| \mid \mathcal{E}(F,t,y) \right]\notag\\
    &\hspace{1.1cm} \times p_{\nabla_{v(F)} (X)(\rho_{v(F)}^F(t))\nabla_{v(F)} (X)(\rho_{v(F)}^F(0))}(y^{v(F)},y^{v(F)})\mathcal{H}^{d-m}(W^F_{v(F)} \cap (W^F_{v(F)}-t))\, dt\notag\\
    &\quad\leq\int_{\overline{W^F_{v(F)}-W^F_{v(F)}}}\E\left[ |\det D^2_{v(F)}(X)(\rho_{v(F)}^F(t))\det D^2_{v(F)}(X)(\rho_{v(F)}^F(0))| \mid \mathcal{E}(F,t,y) \right]\notag\\
    &\hspace{1.5cm} \times p_{\nabla_{v(F)} (X)(\rho_{v(F)}^F(t))\nabla_{v(F)} (X)(\rho_{v(F)}^F(0))}(y^{v(F)},y^{v(F)})\mathcal{H}^{d-m}(W^F_{v(F)}\cap (W^F_{v(F)}-t)) \,dt,
  \end{align}
  where $\mathcal{E}(F,t,y)$ denotes the event $\{\nabla _{v(F)}(X)(\rho_{v(F)}^F(t))=\nabla _{v(F)}(X)(\rho_{v(F)}^F(0))=y^{v(F)}\}$.
  To obtain the finiteness of the latter integral, we apply Lemmata \ref{lem:boundDensity} and \ref{lem:boundExp} from the appendix with $N:=2\sup\{\|x\|\mid x\in W\}$, which provide an integrable upper bound for the integrand. Note that all constants, appearing in these lemmata, are independent of $F\in A(d,d-m)$ and we therefore obtain the assertion.

  We continue with the proof of (ii), which uses the ideas in \cite [Proposition 1.1 (2)]{paper:EstradeLeon}. We abbreviate
  \begin{align*}
      G(F,t,y)&:=\E\left[ |\det D^2_{v(F)}(X)(\rho_{v(F)}^F(t))\det D^2_{v(F)}(X)(\rho_{v(F)}^F(0))| \mid \mathcal{E}(F,t,y) \right]\\
      &\quad \times p_{\nabla_{v(F)} (X)(\rho_{v(F)}^F(t))\nabla_{v(F)} (X)(\rho_{v(F)}^F(0))}(y^{v(F)},y^{v(F)}).
  \end{align*}
  The application of Rice's formula in (\ref{eq:rice}) shows that the first moment of the counting variable $\# \{t\in W \cap F \colon \nabla (X|_F)(t)=y\}$ is continuous in $y$. Thus it remains to show the continuity of the second factorial moment, which can be written, using Rice's formula, cf. equation (\ref{eq:ricefac}), as
  \begin{align*}
    \varphi(F,y):=&\int_{W^F_{v(F)}-W^F_{v(F)}}G(F,t,y)\mathcal{H}^{d-m}(W^F_{v(F)}\cap (W^F_{v(F)}-t))\, dt.
  \end{align*}
  Lemmata \ref{lem:boundDensity} and \ref{lem:boundExp} yield that for any number $\eta>0$ there exists a number $\varepsilon >0$ such that
  \begin{align*}
    \int_{B^{d-m}_{\varepsilon}}\E\left[ |\det D^2_{v(F)}(X)(\rho_{v(F)}^F(t))\det D^2_{v(F)}(X)(\rho_{v(F)}^F(0))| \mid \mathcal{E}(F,t,y) \right]&\\
   \times p_{\nabla_{v(F)} (X)(\rho_{v(F)}^F(t))\nabla_{v(F)}(X)(\rho_{v(F)}^F(0))}(y^{v(F)},y^{v(F)}) \,dt &< \eta,
  \end{align*}
  uniformly in $y\in F^\circ\cap D$. Thus for $y,y'\in F^\circ \cap D$
  \begin{align}\label{eq:conti}
    |\varphi(F,y)-\varphi(F,y')| \leq 2c\eta + c&\int_{W^F_{v(F)}-W^F_{v(F)}\setminus B^{d-m}_\varepsilon}| G(F,t,y)-G(F,t,y')| dt.
  \end{align}
  Observe now that by Gaussian regression, cf. \cite [Proposition 1.2]{book:AzaisWschebor}, we obtain
  \begin{align*}
    \E&\left[ |\det D^2_{v(F)}(X)(\rho_{v(F)}^F(t))\det D^2_{v(F)}(X)(\rho_{v(F)}^F(0))| \mid \mathcal{E}(F,t,y)\right]
    \\&=\E\left[ |\det (A(F,t)y^{v(F)}+Z(F,t))\det ( A'(F,t)y^{v(F)}+Z'(F,t))| \right]
    \\&= \E\left[ \Big | \det \Big (Z(F,t)_{(k-1)(d-m)+l}+\sum_{i=1}^{d-m}A_{(k-1)(d-m)+l ,i}(F,t)y^{v(F)}_{i}\Big)_{k,l=1}^{d-m}\right.\\
    &\hspace{1cm} \times \left .\det\Big (Z'(F,t)_{(k-1)(d-m)+l}+\sum_{i=1}^{d-m}A'_{(k-1)(d-m)+l ,i}(F,t)y^{v(F)}_{i}\Big)_{k,l=1}^{d-m} \Big |\right],
  \end{align*}
  for suitable matrices $A(F,t)$ and $A'(F,t)$, random vectors $Z(F,t)$ and $ Z'(F,t)$. Thus for a difference of the above expression in $y$ and $y'$, we can use the reverse triangle inequality and expand the determinant, to obtain the continuity of the conditional expectation in $y$. Hence, $G(F,t,\cdot)$ is continuous as a product of continuous functions.

  To apply dominated convergence for $y'\to y$ in (\ref{eq:conti}), observe that Lemmata \ref{lem:boundDensity} and \ref{lem:boundExp} yield for $t\in \overline{W^F_{v(F)}}$ the estimates
  \begin{align*}
    \E&\left[ |\det D^2_{v(F)}(X)(\rho_{v(F)}^F(t))\det D^2_{v(F)}(X)(\rho_{v(F)}^F(0))| \mid \mathcal{E}(F,t,y)\right]\\
    &\leq c \|t\|^2(c+c\|y\|^{4(d-m-1)})^{1/2}(c+c\|y\|^{4})^{1/2}\\
    &\leq c \|t\|^2\sup_{y\in D}(c+c\|y\|^{4(d-m-1)})^{1/2}(c+c\|y\|^{4})^{1/2}
  \end{align*}
  and
  \begin{align*}
    p_{\nabla_{v(F)} (X)(\rho_{v(F)}^F(t))\nabla_{v(F)} (X)(\rho_{v(F)}^F(0))}(y^{v(F)},y^{v(F)}) \leq c\|t\|^{-(d-m)}\ind_U(t)+c\ind_{U}(t),
  \end{align*}
  where $c>0$ is a constant independent of $t$ and $y$, and $U$ is a neighborhood of $0$. Thus we obtain an integrable upper bound for $|G(F,t,y)-G(F,t,y')|$ independent of $y'$. We conclude
  \begin{align*}
    \lim_{y'\to y} |\varphi(F,y)-\varphi(F,y')| \leq 2c\eta + c\int_{W^F_{v(F)}-W^F_{v(F)}\setminus B^{d-m}_\varepsilon}\lim_{y'\to y}| G(F,t,y)-G(F,t,y')|\, dt = 2c\eta.
  \end{align*}
  Taking the limit $\eta\to 0$ yields the assertion.

  We conclude part (iii) by following the lines of \cite [Proposition 1.2]{paper:EstradeLeon}. We first show that for $F\in A(d,d-m)$
  \begin{align*}
    \int_{W\cap F}& \delta_\varepsilon(\nabla (X|_F)(t))|\det(D^2(X|_F)(t))| \,dt\overset{L^2(\mathbb{P})}{\underset{\varepsilon \to 0}{\longrightarrow}} \# \{t\in W \cap F \colon \nabla (X|_F)(t)=0\}.
  \end{align*}
  Note that by the same proof as used for Lemma \ref{lem:trafo}, whose preliminaries, for fixed $F\in A(d,d-m)$, are checked in \cite [Lemma 11.2.10 - 11.2.12]{book:AdlerTaylor}, we obtain almost sure convergence. Thus we obtain by an application of Fatou's Lemma
  \begin{align*}
    \E&\left[ \# \{t\in W \cap F \colon \nabla (X|_F)(t)=0\}^2\right]
    \\&\leq \liminf_{\varepsilon \to 0} \E\left[ \left (\int_{W\cap F}\delta_\varepsilon(\nabla (X|_F)(t))|\det(D^2(X|_F)(t))| \,dt\right)^2\right]\\
    &\leq \limsup_{\varepsilon \to 0} \E\left[ \left(\int_{W\cap F}\delta_\varepsilon(\nabla (X|_F)(t))|\det(D^2(X|_F)(t))| \,dt\right)^2\right].
  \end{align*}
  An application of the coarea formula, cf. \cite [Theorem 3.2.12]{book:Federer}, yields
  \begin{align*}
    \int_{W\cap F}&\delta_\varepsilon(\nabla (X|_F)(t))|\det(D^2(X|_F)(t))| dt= \int_{F^\circ} \#\{t\in W\cap F : \nabla(X|_F)(t)=y\}\delta_\varepsilon(y)\,dy,
  \end{align*}
  which leads to the upper bound
  \begin{align*}
    \limsup_{\varepsilon \to 0} \E&\left[ \left (\int_{F^\circ} \#\{t\in W\cap F : \nabla(X|_F)(t)=y\}\delta_\varepsilon(y)\,dy\right)^2\right] \\
    &\leq \limsup_{\varepsilon \to 0} \E\left[ \int_{F^\circ} \#\{t\in W\cap F : \nabla(X|_F)(t)=y\}^2\delta_\varepsilon(y)\,dy\right]\\
    &=\limsup_{\varepsilon \to 0} \int_{F^\circ} \E\left[\#\{t\in W\cap F : \nabla(X|_F)(t)=y\}^2\right]\delta_\varepsilon(y)\,dy\\
    &=\E\left[\#\{t\in W\cap F : \nabla(X|_F)(t)=0\}^2\right],
  \end{align*}
  where we used Jensen's inequality in the second line and point (ii) in the last. This shows the assertion. Together with the fact that
  \begin{align*}
    |\xi_W(F,\varepsilon)|\leq \int_{W\cap F} \delta_\varepsilon(\nabla (X|_F)(t))|\det(D^2(X|_F)(t))|\, dt
  \end{align*}
  and Lemma \ref{lem:trafo}, whose assumptions, for fixed $F$, are again checked in \cite [Section 11.2]{book:AdlerTaylor}, we conclude the lemma by a variant of the dominated convergence theorem, cf. \cite [Theorem 1.20]{book:EvansGariepy}, and note that especially $\xi_W(F)\in L^2(\mathbb{P})$. \qedhere
\end{proof}

\begin{lemma}\label{lem:detformula}
  For real numbers $c_1,c_2\in \R$ and $v\in \R^d$ define the matrix $A:=c_1I_{d}+c_2 vv^\top$. Then
  \begin{align*}
    \det A = c_1^{d}+c_1^{d-1}c_2\|v\|^2.
  \end{align*}
\end{lemma}
\begin{proof}
  Note that for $c\in\R$ and $u\in v^\perp$, we obtain
  \begin{align*}
    (I_{d}+cvv^\top)v=(1+c\|v\|^2)v \quad\text{ and }\quad (I_{d}+cvv^\top)u =u.
  \end{align*}
  Thus the linear mapping associated with the matrix $I_{d}+cvv^\top$ has the eigenvalues $1+c\|v\|^2,1,\ldots,1$, yielding
  \begin{align*}
    \det(I_{d}+cvv^\top)= 1+c\|v\|^2.
  \end{align*}
  Hence by choosing $c:=c_2/c_1$ --- for $c_1=0$ the lemma holds trivially --- we obtain
  \begin{equation*}
    \det(A)= c_1^{d}\det(I_{d}+cvv^\top)= c_1^{d}+c_1^{d-1}c_2\|v\|^2. \qedhere
  \end{equation*}
\end{proof}

\begin{lemma}\label{lem:boundDensity}
  There is a constant $c$, depending on the covariance structure of $X$, $d$, $m$ and $N$, and an open neighborhood $U\subset \R^{d-m}$ of $0$, such that for $F\in A(d,d-m)$, $t\in U$ and $y\in F^\circ$
  \begin{align*}
    p_{\nabla_{v(F)} (X)(\rho_{v(F)}^F(t))\nabla_{v(F)}(X)(\rho_{v(F)}^F(0))}(y^{v(F)},y^{v(F)}) \leq c \|t\|^{-(d-m)}.
  \end{align*}
  Furthermore, there exists a constant $c$, depending on the covariance structure of $X$, $d$, $m$ and $N$, such that for $F\in A(d,d-m)$, $t\in \overline{B^{d-m}_{2N}}\setminus U$ and $y\in F^\circ$
  \begin{align*}
    p_{\nabla_{v(F)} (X)(\rho_{v(F)}^F(t))\nabla_{v(F)} (X)(\rho_{v(F)}^F(0))}(y^{v(F)},y^{v(F)}) \leq c.
  \end{align*}
\end{lemma}

\begin{proof}
  Note first that
  \begin{align*}
    p_{\nabla_{v(F)} (X)(\rho_{v(F)}^F(t))\nabla_{v(F)} (X)(\rho_{v(F)}^F(0))}(y^{v(F)},y^{v(F)}) \leq p_{\nabla_{v(F)} (X)(\rho_{v(F)}^F(t))\nabla_{v(F)} (X)(\rho_{v(F)}^F(0))}(0,0),
  \end{align*}
  since $\left(\nabla_{v(F)} X(\rho_{v(F)}^F(t)),\nabla_{v(F)} X(\rho_{v(F)}^F(0))\right)$ is a $2(d-m)$-dimensional, centered Gaussian vector. In order to bound the right side, we have to bound the determinant of this vector's covariance matrix from below. The covariance matrix is given by
  \begin{align*}
    \begin{pmatrix}
      I_{d-m} & \left( -\frac{\partial^2}{\partial v_i \partial v_j}(\on{Cov}^X)(\sigma^{v(F)}(t))\right)_{i,j=1}^{d-m} \\
      \left( -\frac{\partial^2}{\partial v_i \partial v_j}(\on{Cov}^X)(\sigma^{v(F)}(t))\right)_{i,j=1}^{d-m} & I_{d-m}
    \end{pmatrix},
  \end{align*}
  by \eqref{eq:covFirstDer} and $\E\left[ \frac{\partial}{\partial v_i}X(t)\frac{\partial}{\partial v_j}X(t')\right]= -\frac{\partial^2}{\partial v_i \partial v_j}(\on{Cov}^X)(t-t')$ for $i,j=1,\ldots, d-m$ and $t,t' \in \R^d$. Using \cite [2.8.4]{book:Bernstein}, the determinant of this matrix equals
  \begin{align*}
    \det\left (I_{d-m} - \left( \frac{\partial^2}{\partial v_i \partial v_j}(\on{Cov}^X)(\sigma^{v(F)}(t))\right)_{i,j=1,\ldots,d-m}^2 \right ).
  \end{align*}
  By isotropy, stationarity and the differentiability of $X$, cf. \ref{as:dif}, there exists $R\colon [0,\infty)\to \R, r \mapsto \on{Cov}^X(re_1)$ of class $\mathcal{C}^6$, such that $\on{Cov}^X(t)=R(\|t\|)$ for $t\in \R^d$. Differentiating this identity yields for $t\in \R^{d}\setminus \{0\}$
  \begin{align*}
    D^2\Co^X(t)= R'(\|t\|)\|t\|^{-1}I_d + (R''(\|t\|)\|t\|^{-2}-R'(\|t\|)\|t\|^{-3})(t_it_j)_{i,j=1}^d
  \end{align*}
  and we obtain
  \begin{align*}
    D^2_{v(F)}\Co^X(t)&= \begin{pmatrix}
                          v_1| & \cdots & |v_{d-m}\notag \\
                       \end{pmatrix}^\top
                       D^2\Co^X(t)
                       \begin{pmatrix}
                          v_1| & \cdots & |v_{d-m} \\
                       \end{pmatrix}
    \\&=R'(\|t\|)\|t\|^{-1}I_{d-m} + (R''(\|t\|)\|t\|^{-2}-R'(\|t\|)\|t\|^{-3})(\langle v_i,t\rangle\langle v_j, t \rangle)_{i,j=1}^{d-m}.
  \end{align*}
  Thus for $0\neq t\in \R^{d-m}$, we conclude
  \begin{align}\label{eq:D^2_{v(F)}}
    D^2_{v(F)}(\Co^X)(\sigma^{v(F)}(t))=R'(\|t\|)\|t\|^{-1}I_{d-m} + (R''(\|t\|)\|t\|^{-2}-R'(\|t\|)\|t\|^{-3})(t_it_j)_{i,j=1}^{d-m}.
  \end{align}
  Note that the right side is independent of the specific affine subspace $F$ as a result of the rotational invariance of $\Co^X$.
  Moreover, note that $\left ((t_it_j)_{i,j=1}^{d-m}\right )^2 = \|t\|^2(t_it_j)_{i,j=1}^{d-m}$, for $t\in \R^{d-m}$, and thus
  \begin{align}\label{eq:Iso}
    D^2_{v(F)}&(\Co^X)(\sigma^{v(F)}(t))^2 = (R'(\|t\|)\|t\|^{-1})^2 I_{d-m}+ (R''(\|t\|)\|t\|^{-2}-R'(\|t\|)\|t\|^{-3})\notag \\
    & \quad \times(2R'(\|t\|)\|t\|^{-1} + \|t\|^2(R''(\|t\|)\|t\|^{-2}-R'(\|t\|)\|t\|^{-3}))(t_it_j)_{i,j=1}^{d-m}\notag\\
    &=(R'(\|t\|)\|t\|^{-1})^2 I_{d-m}+\|t\|^{-2} (R''(\|t\|)^2-(R'(\|t\|)\|t\|^{-1})^{2}) (t_it_j)_{i,j=1}^{d-m} .
  \end{align}
  We now apply Lemma \ref{lem:detformula} to establish for the determinant
  \begin{align}\label{eq:determinant}
    \det&\left (I_{d-m} - D^2_{v(F)}(\Co^X)(\sigma^{v(F)}(t))^2 \right )\notag\\
    &=(1-(R'(\|t\|)\|t\|^{-1})^2)^{d-m}-(1-(R'(\|t\|)\|t\|^{-1})^2)^{d-m-1}\notag\\
    &\quad \times\|t\|^{-2}(R''(\|t\|)^2-(R'(\|t\|)\|t\|^{-1})^2)\|t\|^2
    \notag\\
    &=(1-(R'(\|t\|)\|t\|^{-1})^2)^{d-m}-(1-(R'(\|t\|)\|t\|^{-1})^2)^{d-m-1}(R''(\|t\|)^2-(R'(\|t\|)\|t\|^{-1})^2)\notag\\
    &=(1-(R'(\|t\|)\|t\|^{-1})^2)^{d-m-1}(1-R''(\|t\|)^2).
  \end{align}
  Hence the determinant is independent of $F\in A(d,d-m)$ and continuous in $t\in \R^{d-m}\setminus \{0\}$. Therefore, we can bound the density independently of $F$ and $y$ for $t\in \overline{B^{d-m}_{2N}}\setminus U$, where $U\subset \R^{d-m}$ is a open set containing $0$.

  We now prove the asserted estimate for a neighborhood of $0$ and therefore use Taylor's theorem twice, to obtain the expansions up to the fifth derivative
  \begin{align*}
    R'(r)=\sum_{k=0}^4 \frac{R^{(k+1)}(0)}{k!}r^k + o(r^4) \quad\text{ and }\quad
    R''(r)=\sum_{k=0}^3 \frac{R^{(k+2)}(0)}{k!}r^k + o(r^3),
  \end{align*}
  for $r \to 0$. Note that due to \ref{as:dif}, we have $R''(0)=-1$ and moreover, odd derivatives of $R$ vanish at $0$ due to the stationarity of $X$, cf. \cite [(5.5.3),(5.5.5)]{book:AdlerTaylor}. We therefore obtain
  \begin{align}\label{eq:Taylorexp}
    R'(\|t\|)&=-\|t\| + \frac{\mu}{3!}\|t\|^3 + o(\|t\|^4)\quad \text{ and }\quad
    R''(\|t\|)= -1 + \frac{\mu}{2}\|t\|^2+o(\|t\|^3)
  \end{align}
  for $\|t\| \to 0$, where $\mu:=\E\left[ \frac{\partial^2}{\partial t_1\partial t_1}X(0)^2\right]>0$ by \ref{as:nondegenerate}.
  We conclude with equation (\ref{eq:determinant})
  \begin{align*}
    \det&\left (I_{d-m} - D^2_{v(F)}(\Co^X)(\sigma^{v(F)}(t))^2 \right )=3\left(\frac{\mu}{3}\right)^{d-m}\|t\|^{2(d-m)}+o(\|t\|^{2(d-m)+1}),
  \end{align*}
  for $\|t\|\to 0$ and uniformly in $F$. We therefore find a neighborhood $U\subset \R^{d-m}$ of $0$ and a constant $c>0$ independent of $F$ and $y$, such that
  \begin{align*}
    \det&\left (I_{d-m} - D^2_{v(F)}(\Co^X)(\sigma^{v(F)}(t))^2 \right ) \geq c \|t\|^{2(d-m)},
  \end{align*}
  for $t\in U$. From this estimate, the asserted bound
  \begin{align}\label{eq:density}
    p_{\nabla_{v(F)} (X)(\rho_{v(F)}^F(t))\nabla_{v(F)} (X)(\rho_{v(F)}^F(0))}(0,0) &= (2\pi)^{-(d-m)}\det \left (I_{d-m} - D^2_{v(F)}(\Co^X)(\sigma^{v(F)}(t))^2\right)^{-\frac{1}{2}}\notag \\
    &\leq c \|t\|^{-(d-m)}
  \end{align}
  follows. \qedhere
\end{proof}

\begin{lemma} \label{lem:boundExp}
  There is a constant $c$, depending on the covariance structure of $X$, $N$, $d$ and $m$, such that for $F\in A(d,d-m)$, where $F\cap B^d_N\neq \emptyset$, $t\in \overline{B^{d-m}_{2N}}$ and $y\in F^\circ$
 \begin{align*}
   \E&\left[ |\det D^2_{v(F)}(X)(\rho_{v(F)}^F(t))\det D^2_{v(F)}(X)(\rho_{v(F)}^F(0))| \mid \mathcal{E}(F,t,y) \right]  \\
   &\leq c\|t\|^{2}(1+\|y\|^{4(d-m-1)})^{\frac{1}{2}}(1+\|y\|^4)^{\frac{1}{2}}.
  \end{align*}
\end{lemma}

\begin{proof}
  We start with an application of the Cauchy-Schwarz inequality to obtain for $t\in \R^{d-m}$
  \begin{align*}
    \E&\left[ |\det D^2_{v(F)}(X)(\rho_{v(F)}^F(t))\det D^2_{v(F)}(X)(\rho_{v(F)}^F(0))| \mid \mathcal{E}(F,t,y) \right] \notag
    \\&\leq \E\left[ \det (D^2_{v(F)}(X)(\rho_{v(F)}^F(t)))^2 \mid \mathcal{E}(F,t,y) \right]^{\frac{1}{2}}\notag
    \E\left[ \det (D^2_{v(F)}(X)(\rho_{v(F)}^F(0)))^2 \mid \mathcal{E}(F,t,y) \right]^{\frac{1}{2}}\notag\\
    &=\E\left[ \det (D^2_{v(F)}(X)(\rho_{v(F)}^F(0)))^2 \mid \mathcal{E}(F,-t,y) \right]^{\frac{1}{2}}\E\left[ \det (D^2_{v(F)}(X)(\rho_{v(F)}^F(0)))^2 \mid \mathcal{E}(F,t,y) \right]^{\frac{1}{2}}
  \end{align*}
  where we used stationarity in the last equation. We bound the right factor by a bound solely depending on the norm of $t$, hence giving a bound for the left one as well.

  We first use Hadamard's inequality, cf. \cite [Fact 8.17.11]{book:Bernstein}, which reads: For a symmetric and positive semidefinite matrix $A\in \R^{(d-m)\times (d-m)}$ and an orthonormal basis $(u_1,\ldots, u_{d-m})$ of $\R^{d-m}$, we have that
  \begin{align*}
    \det (A) \leq \prod_{i=1}^{d-m} \langle A u_i,u_i\rangle.
  \end{align*}
  Note that for $t\in \R^{d-m}$, we have
  \begin{align*}
    t^\top D^2_{v(F)}(X)(\rho_{v(F)}^F(0))^2 t = (D^2_{v(F)}(X)(\rho_{v(F)}^F(0))t)^\top D^2_{v(F)}(X)(\rho_{v(F)}^F(0))t \geq 0
  \end{align*}
  and therefore we obtain for $t\in \R^{d-m}\setminus\{0\}$ and a suitable choice of $(u_2,\ldots,u_{d-m})$
  \begin{align}\label{eq:detinequality}
     \det (D^2_{v(F)}(X)(\rho_{v(F)}^F(0)))^2\notag &\leq \left \langle D^2_{v(F)}(X)(\rho_{v(F)}^F(0))^2\frac{t}{\|t\|}, \frac{t}{\|t\|}\right \rangle \prod_{i=2}^{d-m} \langle D^2_{v(F)}(X)(\rho_{v(F)}^F(0))^2u_i,u_i \rangle\notag\\
     &\leq \|t\|^{-2}\| D^2_{v(F)}(X)(\rho_{v(F)}^F(0))t\|^2 \| D^2_{v(F)}(X)(\rho_{v(F)}^F(0))\|^{2(d-m-1)},
  \end{align}
  where we used the symmetry of $D^2_{v(F)}(X)(\rho_{v(F)}^F(0))$, Cauchy-Schwarz and a matrix norm that is compatible to the Euclidean norm and submultiplicative, e.g. the induced Euclidean norm. Define now the mapping
  \begin{align*}
    Y_t\colon [0,1] \to \R^{d-m}, x \mapsto \nabla_{v(F)}(X)(\rho_{v(F)}^F(xt)),\quad \text{for } t\in \R^{d-m},
  \end{align*}
  to obtain $Y_t(0)=\nabla_{v(F)}(X)(\rho_{v(F)}^F(0))$, $Y_t(1)=\nabla_{v(F)}X(\rho_{v(F)}^F(t))$ and
  \begin{align*}
    Y'_t(x)= D((\frac{\partial}{\partial v_i}X)_{i=1}^{d-m} \circ \rho_{v(F)}^F \circ \cdot t ) (x)= D^2_{v(F)}(X) (\rho_{v(F)}^F(xt))t,
  \end{align*}
  thus
  $Y'_t(0)=D^2_{v(F)}(X)(\rho_{v(F)}^F(0))t$. Calculating the second derivative of $Y_t$ yields for the $j$-th component of $Y''_t$, $j=1,\ldots,d-m$,
  \begin{align*}
    Y_{t,j}''(x) = \sum_{i=1}^{d}\frac{\partial}{\partial x} \frac{\partial^2}{\partial v_jv_i} (X)(\rho_{v(F)}^F(xt)) t_i= \left \langle \frac{\partial}{\partial v_j} D^2_{v(F)}(X) (\rho_{v(F)}^F(xt))t, t \right\rangle,
  \end{align*}
  where $\frac{\partial}{\partial v_j} D^2_{v(F)}X:=\left( \frac{\partial}{\partial v_j}\frac{\partial^2}{\partial v_i\partial v_j}X\right)_{i,j=1}^{d-m}
  $ Using Taylor's theorem for the mapping $Y$ at $0$ and evaluating the expansion at $1$, yields
  \begin{align*}
    Y_t(1)=Y_t(0) + Y'_t(0)+ \frac{1}{2}\left(Y_{t,1}''(\xi_1),\ldots,Y_{t,d-m}''(\xi_{d-m})\right),
  \end{align*}
  for suitable points $\xi_1,\ldots,\xi_{d-m}\in [0,1]$. Conditioning of the latter equation on the event given by $\mathcal{E}(F,t,y)=\{\nabla _{v(F)}(X)(\rho_{v(F)}^F(t))=\nabla _{v(F)}(X)(\rho_{v(F)}^F(0))=y^{v(F)}\}$ leads to
  \begin{align*}
    D^2_{v(F)}(X) (\rho_{v(F)}^F(0))t =-\frac{1}{2}\left (\left \langle \frac{\partial}{\partial v_i} D^2_{v(F)}(X) (\rho_{v(F)}^F(\xi_it))t, t \right \rangle \right)_{i=1}^{d-m}.
  \end{align*}
  By taking norms, an application of the Cauchy-Schwarz inequality in every component and the compatibility of the matrix norm, we obtain
  \begin{align*}
    \|D^2_{v(F)}(X) (\rho_{v(F)}^F(0))t\|^2 &\leq \frac{1}{4} \|t\|^4 \sum_{i=1}^{d-m}\left \| \frac{\partial}{\partial v_i} D^2_{v(F)}(X) (\rho_{v(F)}^F(\xi_it))\right \|^2
    \\&\leq \frac{1}{4} \|t\|^4 \sum_{i=1}^{d-m} \sup_{x\in [0,1]} \left \| \frac{\partial}{\partial v_i} D^2_{v(F)}(X) (\rho_{v(F)}^F(xt))\right \|^2.
  \end{align*}
  Hence, we conclude with (\ref{eq:detinequality}) and several applications of the Cauchy-Schwarz inequality in the last line
  \begin{align} \label{eq:detinequality2}
    &\E\left[ \det (D^2_{v(F)}(X)(\rho_{v(F)}^F(0)))^2 \mid \mathcal{E}(F,t,y) \right]\notag
    \\&\leq c\|t\|^{-2}\E\left[\|D^2_{v(F)}X(\rho_{v(F)}^F(0))t\|^2\|D^2_{v(F)}(X) (\rho_{v(F)}^F(0))\|^{2(d-m-1)}\mid \mathcal{E}(F,t,y)\right]\notag
    \\&\leq c\|t\|^2\sum_{i=1}^{d-m} \E\left[\sup_{x\in [0,1]} \left \| \frac{\partial}{\partial v_i} D^2_{v(F)}(X) (\rho_{v(F)}^F(xt))\right \|^2\|D^2_{v(F)}(X)(\rho_{v(F)}^F(0))\|^{2(d-m-1)}\mid \mathcal{E}(F,t,y) \right]\notag
    \\&\leq c\|t\|^2 \E\left[ \|D^2_{v(F)}(X)(\rho_{v(F)}^F(0))\|^{4(d-m-1)}\mid \mathcal{E}(F,t,y)\right]^{\frac{1}{2}}\notag
    \\&\quad \times\sum_{i=1}^{d-m}\E\left[\sup_{x\in [0,1]} \left \| \frac{\partial}{\partial v_i} D^2_{v(F)}(X) (\rho_{v(F)}^F(xt))\right \|^4 \mid \mathcal{E}(F,t,y)\right]^{\frac{1}{2}}.
  \end{align}
  Invoking the Lemmata \ref{lem:boundedExp1} and \ref{lem:boundedExp2}, we conclude the proof. \qedhere
\end{proof}

Before we prove the upper bound for the conditional expectations, we prove the following auxiliary lemma.
\begin{lemma}\label{lem:matriceBounds}
  There exists a constant $c>0$ depending on $X$, $d$, $m$ and $N$, such that for all $F\in A(d,d-m)$, where $F\cap B^d_N\neq \emptyset$, $t\in \overline{B^{d-m}_{2N }}$ and $\alpha,\beta,\gamma = 1, \ldots, d$, we have that
  \begin{align*}
    \left\|\left(I_{d-m}-D^2_{v(F)} (\Co^X)(\sigma^{v(F)}(t))\right)^{-1}\right\| &\leq c,\\
    \left \|\left(\frac{\partial^3}{\partial t_\alpha \partial t_\beta\partial v_i}(\Co^X)(\sigma^{v(F)}(t))\right)_{i=1}^{d-m}\right \|^2 \left \|\left (I_{d-m}- D^2_{v(F)}(\Co^X)(\sigma^{v(F)}(t))^2\right)^{-1}\right\| &\leq c,  \\
    \sup_{x\in[0,1]}\left \| \left (-\frac{\partial^4}{\partial t_\alpha \partial t_\beta \partial t_\gamma \partial v_i }(\Co^X)(\sigma^{v(F)}(xt)) \right)_{i=1}^{d-m}\right.\\
    + \left.\left(\frac{\partial^4}{\partial t_\alpha \partial t_\beta \partial t_\gamma \partial v_i }(\Co^X)(\sigma^{v(F)}((x-1)t)) \right)_{i=1}^{d-m} \right \|\left \|\left (I_{d-m}+ D^2_{v(F)}(\Co^X)(\sigma^{v(F)}(t))\right)^{-1}\right\| &\leq c.
  \end{align*}
\end{lemma}

\begin{proof}
  We distinguish the case $t\in \overline{B^{d-m}_{2N}}\setminus U$, where $U$ is open and contains $0$, and the case in which $t\in U$. Note that for the different inequalities $U$ may be chosen different and we think of the matrix norm as the one, which suits us most, knowing that we can bound one by a multiple of the other.

  We start with $t\in \overline{B^{d-m}_{2N}}\setminus U$ and think of the norm as the spectral norm. Observe that in this case
  \begin{align*}
    \|A^{-1}\|= |\lambda_{\min}(A)|^{-1},
  \end{align*}
  where $A$ is an invertible, symmetric matrix and $\lambda_{\min}(A)$ denotes the eigenvalue of $A$ with smallest absolute value. Furthermore, we see by Lemma \ref{lem:detformula} and equation (\ref{eq:D^2_{v(F)}}), resp. equation (\ref{eq:Iso}), that the coefficients of the polynomials in $\lambda$
  \begin{align*}
    \det&(I_{d-m}-D^2_{v(F)}(\Co^X)(\sigma^{v(F)}(t)) - \lambda I_{d-m}),\\
    \det&(I_{d-m}-D^2_{v(F)}(\Co^X)(\sigma^{v(F)}(t))^2-\lambda I_{d-m}),\\
    \det&(I_{d-m}+D^2_{v(F)}(\Co^X)(\sigma^{v(F)}(t)) - \lambda I_{d-m}),
  \end{align*}
  are independent of $F$ but continuous in $t\in \overline{B^{d-m}_{2N}}\setminus U$. Due to \ref{as:nondegenerate}, we know that
  \begin{align*}
    0&\neq \det \Co\left(\nabla_{v(F)}X(0),\nabla_{v(F)}X(\sigma^{v(F)}(t))\right)=\det\left ( I_{d-m}-(D^2_{v(F)}(\Co^X)(\sigma^{v(F)}(t)))^2\right)\\
    &=\det\left( I_{d-m}-D^2_{v(F)}(\Co^X)(\sigma^{v(F)}(t))\right ) \det\left ( I_{d-m}+D^2_{v(F)}(\Co^X)(\sigma^{v(F)}(t))\right ),
  \end{align*}
  for $t\neq 0$ and therefore none of the involved matrices has eigenvalue $0$. And since the zeros of a polynomial are continuous in the coefficients, we conclude that the norms
  \begin{align*}
    \|(I_{d-m}-D^2_{v(F)}(\Co^X)(\sigma^{v(F)}(t)))^{-1}\|,\\
    \|(I_{d-m}-D^2_{v(F)}(\Co^X)(\sigma^{v(F)}(t))^2)^{-1}\|,\\
    \|(I_{d-m}+D^2_{v(F)}(\Co^X)(\sigma^{v(F)}(t)))^{-1}\|
  \end{align*}
  are bounded for $t\in \overline{B^{d-m}_{2N}}\setminus U$, independently of $F$.
  In order to bound the supremum of the norm $\left \|(-\frac{\partial^4}{\partial t_\alpha \partial t_\beta \partial t_\gamma \partial v_i }(\Co^X)(\sigma^{v(F)}(xt)) )_{i=1}^{d-m}+(\frac{\partial^4}{\partial t_\alpha \partial t_\beta \partial t_\gamma \partial v_i }(\Co^X)(\sigma^{v(F)}((x-1)t)) )_{i=1}^{d-m} \right\|$ for $x\in[0,1]$ as well as the norm $\left\|\left(\frac{\partial^3}{\partial t_\alpha \partial t_\beta\partial v_i}(\Co^X)(\sigma^{v(F)}(t))\right)_{i=1}^{d-m}\right\|$, we  bound the directional derivatives by the partial ones and use the continuity, as shown exemplarily in the following:
  \begin{align*}
  \left \|\left(\frac{\partial^3}{\partial t_\alpha \partial t_\beta\partial v_i}(\Co^X)(\sigma^{v(F)}(t))\right)_{i=1}^{d-m}\right \|^2&=\sum_{i=1}^{d-m}\frac{\partial^3}{\partial t_\alpha \partial t_\beta \partial v_i}(\Co^X)(\sigma^{v(F)}(t))^2
  \\&=\sum_{i=1}^{d-m}\left(\sum_{\gamma=1}^{d}v_i^{(\gamma)} \frac{\partial^3}{\partial t_\alpha \partial t_\beta \partial t_\gamma}(\Co^X)(\sigma^{v(F)}(t))\right)^2,
  \end{align*}
  which can be bounded by
  \begin{align*}
    (d-m)\left(\sum_{\gamma=1}^{d}\sup_{s\in \overline{B^{d}_{2N}}}\left | \frac{\partial^3}{\partial t_\alpha \partial t_\beta \partial t_\gamma}\Co^X(s)\right |\right)^2< \infty,
  \end{align*}
  independently of $F$ and $t$.

  To analyse the behaviour for $t$ near 0, observe that $I_{d-m}-D^2_{v(F)}\Co^X(t)\stackrel{\|t\|\to 0}{\longrightarrow}2I_{d-m}$ and thus $\|(I_{d-m}-D^2_{v(F)}\Co^X(t))^{-1}\|\to \frac{1}{2}$. Hence, there is no singularity at $t=0$ and the norm can easily be bounded using continuity arguments as above. Since $I_{d-m}+D^2_{v(F)}\Co^X(t)\stackrel{\|t\|\to 0}{\longrightarrow}0$, this is different in the other cases. We proceed with the second inequality of the assertion and use the identity (\ref{eq:Iso}) and the Taylor expansion derived in (\ref{eq:Taylorexp}), to obtain for $0\neq t\in \R^{d-m} $ and $\|t\|\to 0$
  \begin{align*}
    \left ( I_{d-m} - D^2_{v(F)}(\Co^X)(\sigma^{v(F)}(t))^2\right)^{-1} = (\Theta(t) + O(\|t\|^4))^{-1},
  \end{align*}
  uniformly in $F$, where $\Theta(t):=\frac{\mu}{3}\|t\|^2I_{d-m} + \frac{2}{3}\mu (t_it_j)_{i,j=1}^{d-m}$. Since, cf. Lemma \ref{lem:detformula},
  \begin{align*}
    \det \Theta (t) = (\mu/3 \|t\|^2)^{d-m}+ (\mu/3\|t\|^2)^{d-m-1}2/3\mu \|t\|^2 \neq 0
  \end{align*}
  for $t\neq 0$, we conclude that $\Theta(t)$ is invertible and we
  denote its inverse by $\Delta(t)$, for $t\neq 0$. Observe that for $\alpha\geq 0$ the identity $\Theta(\alpha t) = \alpha^2 \Theta(t)$ holds and therefore $\Delta(\alpha t)=\alpha^{-2}\Delta(t)$. Thus we obtain
  \begin{align*}
    \left ( I_{d-m} - D^2_{v(F)}(\Co^X)(\sigma^{v(F)}(t))^2\right)^{-1} = \Delta (t) \left(I_{d-m}- O(\|t\|^4)\Delta(t)\right)^{-1}.
  \end{align*}
  Now, we can conclude from \cite [Proposition 9.4.13]{book:Bernstein}, that for a given matrix $A$ with $\|A(t)\|\to 0$ for $\|t\|\to 0$, we have $\|(I-A(t))^{-1}\|\leq1+\|A(t)\|+ o(\|A(t)\|)$. Before we apply this result, observe that
 \begin{align*}
                \sup_{u\in \mathbb{S}^{d-1}}\|\Delta (u)\|
 \end{align*}
  is actually a maximum and moreover independent of $F$. To see this, think of the norm again as the spectral norm and observe by Lemma \ref{lem:detformula}, that the zeros of the polynomial in $\lambda$
  \begin{align*}
    \det(\Theta(u) - \lambda I_{d-m})
  \end{align*}
  are independent of $F$ but continuous in $u\in \mathbb{S}^{d-1}$, from which we conclude the assertion. Thus we obtain
  \begin{align*}
    \| (I_{d-m} - O(\|t\|^4)\Delta(t))^{-1}\| &\leq 1 + \| O(\|t\|^4)\Delta(t) \| + o(\|O(\|t\|^4)\Delta(t)\|)\\
    &=1 + O(\|t\|^2) + o(\|t\|^2)\\
    &= 1+O(\|t\|^2),
  \end{align*}
  for $\|t\|\to 0$, where we used that $\|O(\|t\|^4)\Delta(t)\| = \| O(\|t\|^4)\|t\|^{-2} \Delta(t/\|t\|))\|= O(\|t\|^2)$ and $g\in o(O(f))$ yields $g\in o(f)$. Hence, we conclude
  \begin{align*}
    \| \left ( I_{d-m} - D^2_{v(F)}(\Co^X)(\sigma^{v(F)}(t))^2\right)^{-1}\| &\leq \| \Delta(t)\| (1+O(\|t\|^2))\\
    &= O(\|t\|^{-2})
  \end{align*}
  for $\|t\|\to 0$ and  uniformly in $F$. Taylor's theorem applied to $\frac{\partial ^3}{ \partial t_\alpha \partial t_\beta\partial v_i}(\Co^X)(\sigma^{v(F)}(\cdot))$ yields for $i=1,\ldots,d-m$, $0\neq t\in \R^{d-m}$ and $\xi\in [0,1]$
  \begin{align*}
    \frac{\partial ^3}{ \partial t_\alpha \partial t_\beta\partial v_i}(\Co^X)(\sigma^{v(F)}(t))&= \frac{\partial ^3}{ \partial t_\alpha \partial t_\beta\partial v_i}\Co^X(0) + \sum_{j=1}^{d-m} \frac{\partial}{\partial t_j}(\frac{\partial ^3}{ \partial t_\alpha \partial t_\beta\partial v_i}\Co^X\circ\sigma^{v(F)})(\xi t_j)t_j\\
    &= O(\|t\|),
  \end{align*}
  since $\frac{\partial ^3}{ \partial t_\alpha \partial t_\beta\partial v_i}\Co^X(0)=0$ by stationarity, cf. \cite [Equation (5.5.3)]{book:AdlerTaylor}. Note this equality holds uniformly in $F$, since $\frac{\partial}{\partial t_j}(\frac{\partial ^3}{ \partial t_\alpha \partial t_\beta\partial v_i}\Co^X\circ\sigma^{v(F)}(t))$ can be bounded independently of $F$ for $t\in \overline{B^{d-m}_{2N}}$. Therefore, we conclude
  \begin{align*}
    \left\|\left (\frac{\partial^3}{\partial t_\alpha \partial t_\beta \partial v_i}(\Co^X)(\sigma^{v(F)}(t))\right ) _{i=1}^{d-m}\right\|^2 \left \|(I_{d-m}- D^2_{v(F)}(\Co^X)(\sigma^{v(F)}(t))^2)^{-1}\right\| = O(1),
  \end{align*}
  for $\|t\|\to 0$ and uniformly in $F$.

  To show the last inequality of the assertion, we use identity (\ref{eq:D^2_{v(F)}}) and the Taylor expansion in (\ref{eq:Taylorexp}), to obtain for $0\neq t\in \R^{d-m}$
  \begin{align*}
    I_{d-m}&+D^2_{v(F)}(\Co^X)(\sigma^{v(F)}(t)\\
    &=(1+R'(\|t\|)\|t\|^{-1})I_{d-m} + \|t\|^{-2}(R''(\|t\|)-R'(\|t\|)\|t\|^{-1})(t_it_j)_{i,j=1}^{d-m}\\
    &=\frac{\mu}{6}\|t\|^2I_{d-m}+ \frac{\mu}{3}(t_it_j)_{i,j=1}^{d-m} + O(\|t\|^4)\\
    &=\frac{1}{2}\Theta(t)+O(\|t\|^4).
  \end{align*}
  The same approach as before, yields
  \begin{align*}
    \left \| (I_{d-m}+D^2_{v(F)}(\Co^X)(\sigma^{v(F)}(t)))^{-1}\right \| = O(\|t\|^{-2})
  \end{align*}
  uniformly in $F$.
  Taylor's theorem applied to $-\frac{\partial^4}{\partial t_\alpha \partial t_\beta \partial t_\gamma \partial v_i}(\Co^X)(\sigma^{v(F)}(x\cdot))$, yields for $t\in \overline{B^{d-m}_{2N}}$, $i=1,\ldots,d-m$ and $x\in [0,1]$
  \begin{align*}
    -\frac{\partial^4}{\partial t_\alpha \partial t_\beta \partial t_\gamma \partial v_i}(\Co^X)(\sigma^{v(F)}(xt))  =-\frac{\partial^4}{\partial t_\alpha \partial t_\beta \partial t_\gamma \partial v_i}\Co^X(0)+ O(\|t\|^2),
  \end{align*}
  since $\frac{\partial^5}{\partial t_j \partial t_\alpha \partial t_\beta \partial t_\gamma \partial v_i}\Co^X(0) =0$, $j=1,\ldots,d$, as $X$ is stationary, cf. \cite [Equation (5.5.3)]{book:AdlerTaylor}. We note that this equality holds uniformly in $F$ and $x\in[0,1]$. Analogously, we obtain
  \begin{align*}
    -\frac{\partial^4}{\partial t_\alpha \partial t_\beta \partial t_\gamma \partial v_i}(\Co^X)(\sigma^{v(F)}((x-1)t))  =-\frac{\partial^4}{\partial t_\alpha \partial t_\beta \partial t_\gamma \partial v_i}\Co^X(0)+ O(\|t\|^2)
  \end{align*}
  uniformly in $F$ and $x\in[0,1]$. Thus
  \begin{align*}
    \sup_{x\in[0,1]}& \left \|\left (-\frac{\partial^4}{\partial t_\alpha \partial t_\beta \partial t_\gamma \partial v_i }(\Co^X)(\sigma^{v(F)}(xt)) \right)_{i=1}^{d-m}\right. \\
    &+\left.\left (\frac{\partial^4}{\partial t_\alpha \partial t_\beta \partial t_\gamma \partial v_i }(\Co^X)(\sigma^{v(F)}((x-1)t)) \right)_{i=1}^{d-m} \right \|=O(\|t\|^2)
  \end{align*}
  and therefore the assertion of the lemma follows. \qedhere
\end{proof}

\begin{lemma} \label{lem:boundedExp1}
  There is a constant $c>0$, such that for $F\in A(d,d-m)$, where $F\cap B^d_N \neq \emptyset$,  $t\in \overline{B^{d-m}_{2N}}$ and $y\in F^\circ$
  \begin{align*}
    \E\left[ \|D^2_{v(F)}(X)(\rho_{v(F)}^F(0))\|^{4(d-m-1)}\mid \mathcal{E}(F,t,y)\right] \leq c(1+\|y\|^{4(d-m-1)}).
  \end{align*}
\end{lemma}
Since the proofs of Lemmata \ref{lem:boundedExp1} and \ref{lem:boundedExp2} follow the same idea, we only show the proof of Lemma \ref{lem:boundedExp2}.
\begin{lemma} \label{lem:boundedExp2}
  There is a constant $c>0$, such that for $F\in A(d,d-m)$, where $F\cap B^d_N \neq \emptyset$, $t\in \overline{B^{d-m}_{2N}}$, $y\in F^\circ$, $i\in \{1,\ldots, d-m\}$
  \begin{align*}
    \E\left[ \sup_{x\in[0,1]}\left \|\frac{\partial}{\partial v_i}D^2_{v(F)}(X)(\rho_{v(F)}^F(xt))\right \|^{4}\mid \mathcal{E}(F,t,y)\right] \leq c(1+\|y\|^4).
  \end{align*}
\end{lemma}

\begin{proof}
  We start with the following estimate
  \begin{align*}
    \sup_{x\in[0,1]}\left\|\frac{\partial}{\partial v_i}D^2_{v(F)}(X)(\rho_{v(F)}^F(xt))\right\|^{4}
    &\leq (d-m)^4d^{9}\sum_{\alpha,\beta,\gamma =1}^{d}\sup_{x\in[0,1]} \left |\frac{\partial^3}{\partial t_\alpha\partial t_\beta \partial t_\gamma }(X)(\rho_{v(F)}^F(xt))\right | ^4,
  \end{align*}
   by Jensen's inequality and and the fact that $(v_1,\ldots,v_{d-m})$ is an orthonormal basis. By using Gaussian regression, cf. \cite [Proposition 1.2]{book:AzaisWschebor}, we obtain
  \begin{align*}
    \E&\left[ \sup_{x\in[0,1]} \left |\frac{\partial^3}{\partial t_\alpha\partial t_\beta \partial t_\gamma }(X)(\rho_{v(F)}^F(xt))\right | ^4\mid \mathcal{E}(F,t,y)\right]\\
    &=\E\left[\vphantom{
    \begin{pmatrix}
      y^{v(F)} \\
      y^{v(F)} \\
    \end{pmatrix}}
    \sup_{x\in[0,1]}\left |\frac{\partial^3}{\partial t_\alpha\partial t_\beta \partial t_\gamma }(X)(\rho_{v(F)}^F(xt))-C_{12}^{\alpha,\beta,\gamma}(F,x,t)C_{2}^{-1}(F,t)X_2(F,t)\right.\right. \\
    &\left.\left. \quad+C_{12}^{\alpha,\beta,\gamma}(F,x,t)C_{2}^{-1}(F,t)
    \begin{pmatrix}
      y^{v(F)} \\
      y^{v(F)} \\
    \end{pmatrix}  \right | ^4\vphantom{ \sup_{x\in[0,1]}}\right],
  \end{align*}
  where
  \begin{align*}
  X_2(F,t)&:= (\nabla _{v(F)}(X)(\rho_{v(F)}^F(0)),\nabla _{v(F)}(X)(\rho_{v(F)}^F(t))),\\
  C_{12}^{\alpha,\beta,\gamma}(F,x,t) &:= \Co\left(\frac{\partial^3}{\partial t_\alpha\partial t_\beta \partial t_\gamma }(X)(\rho_{v(F)}^F(xt)), X_2(F,t)\right)\\
  &=\left(K^{\alpha,\beta,\gamma}(F,\sigma^{v(F)}(xt)) , K^{\alpha,\beta,\gamma}(F,\sigma^{v(F)}((x-1)t))\right)\in\R^{1\times 2(d-m)}
  \end{align*}
  with $K^{\alpha,\beta,\gamma}(F,s) := \left(-\frac{\partial^4}{\partial t_\alpha \partial t_\beta \partial t_\gamma \partial v_i}\Co^X(s)\right)_{i=1}^{d-m}$ for $s \in \R^d$ and
  \begin{align*}
    C_{2}(F,t):= \Co(X_2(F,t))\in \R^{2(d-m)\times 2(d-m)}.
  \end{align*}
  Note that $C^{-1}_2(F,t)$ exists due to \ref{as:nondegenerate} and that by \cite [Proposition 2.8.7]{book:Bernstein}
  \begin{align*}
    C_2(F,t)^{-1}=
    \begin{pmatrix}
       A(F,t)& B(F,t) \\
      B(F,t) & A(F,t) \\
    \end{pmatrix},
  \end{align*}
  where
  \begin{align}\label{def:AB}
    A(F,t):&=(I_{d-m}-D^2_{v(F)}(\Co^X)(\sigma^{v(F)}(t))^2)^{-1},\notag\\
    B(F,t):&=-(I_{d-m}-D_{v(F)}^2(\Co^X)(\sigma^{v(F)}(t))^2)^{-1}D^2_{v(F)} (\Co^X)(\sigma^{v(F)}(t)).
  \end{align}
  By the triangle- and Jensen's inequality
  \begin{align*}
    \sup_{x\in[0,1]}&\left |\frac{\partial^3}{\partial t_\alpha\partial t_\beta \partial t_\gamma }(X)(\rho_{v(F)}^F(xt))-C_{12}^{\alpha,\beta,\gamma}(F,x,t)C_{2}^{-1}(F,t)X_2(F,t) \right . \\&\quad\left.+C_{12}^{\alpha,\beta,\gamma}(F,x,t)C_{2}^{-1}(F,t)\begin{pmatrix}
                                                   y^{v(F)} \\
                                                   y^{v(F)} \\
                                                 \end{pmatrix}  \right | ^4 \\
    &\leq 3^3\sup_{x\in[0,1]}\frac{\partial^3}{\partial t_\alpha\partial t_\beta \partial t_\gamma }(X)(\rho_{v(F)}^F(xt))^4+3^3\sup_{x\in[0,1]}|C_{12}^{\alpha,\beta,\gamma}(F,x,t)C_{2}^{-1}(F,t)X_2(F,t) | ^4\\
    &\quad + 3^3\sup_{x\in[0,1]}\left|C_{12}^{\alpha,\beta,\gamma}(F,x,t)C_{2}^{-1}(F,t)\begin{pmatrix}
                                                   y^{v(F)} \\
                                                   y^{v(F)} \\
                                                 \end{pmatrix}\right|^4.
  \end{align*}
  Again the submultiplicativity of the norm and Jensen's inequality yield
  \begin{align*}
    &\sup_{x\in[0,1]}|C_{12}^{\alpha,\beta,\gamma}(F,x,t)C_{2}^{-1}(F,t)X_2(F,t) | ^4\\
    &\quad\leq\sup_{x\in[0,1]}\|C_{12}^{\alpha,\beta,\gamma}(F,x,t)C_{2}^{-1}(F,t)\|^42(d-m)^2d^3\sum_{j=1}^d \left (\sup_{s\in \overline{B^{d}_{N}}}\frac{\partial}{\partial t_j} X(s)^4 +\sup_{s\in \overline{B^{d}_{3N}}}\frac{\partial}{\partial t_j} X(s)^4\right),
  \end{align*}
  where we used in the last line, that $F\cap B^d_N\neq \emptyset$ implies for $t\in \overline{B^{d-m}_{2N}}$ that $\|\rho_{v(F)}^F(t)\|\leq 3N$ holds, as well as $\|\rho_{v(F)}^F(0)\|\leq N$. Using this fact again, and summarizing the estimates, we obtain
  \begin{align*}
    \E&\left[ \sup_{x\in[0,1]}\left |\frac{\partial^3}{\partial t_\alpha\partial t_\beta \partial t_\gamma }(X)(\rho_{v(F)}^F(xt))-C_{12}^{\alpha,\beta,\gamma}(F,x,t)C_{2}^{-1}(F,t)X_2(F,t) \right. \right. \\&\quad\left.\left.+C_{12}^{\alpha,\beta,\gamma}(F,x,t)C_{2}^{-1}(F,t)
    \begin{pmatrix}
      y^{v(F)} \\
      y^{v(F)} \\
    \end{pmatrix}  \right | ^4\right ]\\
    &\leq 3^3 \E\left[ \sup_{s\in \overline{B^{d}_{3N}}}\frac{\partial^3}{\partial t_\alpha\partial t_\beta \partial t_\gamma }X(s)^4\right]
    +3^3\sup_{x\in[0,1]}\|C_{12}^{\alpha,\beta,\gamma}(F,x,t)C_2(F,t)^{-1}\|^42(d-m)^2d^3\\
    &\quad \times\sum_{j=1}^d \left (2\E\left[ \sup_{s\in \overline{B^{d}_{3N}}}\frac{\partial}{\partial t_j} X(s)^4 \right] \right)+3^3\sup_{x\in[0,1]}\left|C_{12}^{\alpha,\beta,\gamma}(F,x,t)C_{2}^{-1}(F,t)\begin{pmatrix}
                                                   y^{v(F)} \\
                                                   y^{v(F)} \\
                                                 \end{pmatrix}\right|^4.
  \end{align*}
  Note that the arguments of the expectations neither depend on $F$ nor on $t$ and moreover, the involved Gaussian fields are all continuous. The continuity implies that it is sufficient to bound the expectation of a supremum of a dense index set and moreover that the necessary conditions in \cite [Theorem 5]{paper:LandauShepp} are satisfied, which guarantees the finiteness of those expectations.

  To proof the Lemma, it remains to bound $\sup_{x\in[0,1]}\|C_{12}^{\alpha,\beta,\gamma}(F,x,t)C_{2}^{-1}(F,t)\|^4$ for $t\in \overline{B^{d-m}_{2N}}$, independently of $F$. Observe that
  \begin{align*}
    \|C_{12}^{\alpha,\beta,\gamma}&(F,x,t)C_{2}^{-1}(F,t)\| \\
    &\leq \| K^{\alpha,\beta,\gamma}(F,\sigma^{v(F)}(xt))A(F,t) +K^{\alpha,\beta,\gamma}(F,\sigma^{v(F)}((x-1)t))B(F,t) \|\\
    &\quad+ \|K^{\alpha,\beta,\gamma}(F,\sigma^{v(F)}(xt))B(F,t) +K^{\alpha,\beta,\gamma}(F,\sigma^{v(F)}((x-1)t))A(F,t)\|
  \end{align*}
  and moreover
  \begin{align*}
    K^{\alpha,\beta,\gamma}&(F,\sigma^{v(F)}(xt))A(F,t) +K^{\alpha,\beta,\gamma}(F,\sigma^{v(F)}((x-1)t))B(F,t)\\
    &=K^{\alpha,\beta,\gamma}(F,\sigma^{v(F)}((x-1)t))(A(F,t)+B(F,t)) \\&\quad +(K^{\alpha,\beta,\gamma}(F,\sigma^{v(F)}(xt))-K^{\alpha,\beta,\gamma}(F,\sigma^{v(F)}((x-1)t)))A(F,t)\\
    &=K^{\alpha,\beta,\gamma}(F,\sigma^{v(F)}((x-1)t))\left (I_{d-m}-D^2_{v(F)}(\Co^X)(\sigma^{v(F)}(t))\right)^{-1}\\
    &\quad +\left (K^{\alpha,\beta,\gamma}(F,\sigma^{v(F)}(xt))-K^{\alpha,\beta,\gamma}(F,\sigma^{v(F)}((x-1)t))\right)\\
     &\quad \times\left (I_{d-m}+D^2_{v(F)}(\Co^X)(\sigma^{v(F)}(t))\right )^{-1}\left (I_{d-m}-D^2_{v(F)}(\Co^X)(\sigma^{v(F)}(t))\right)^{-1},
  \end{align*}
  where we used that $A(F,t)+B(F,t)= (I_{d-m}-D^2_{v(F)}(\Co^X)(\sigma^{v(F)}(t)))^{-1}$. The above equals
  \begin{align*}
    &\left(\vphantom{\left (I_{d-m}+D^2_{v(F)}(\Co^X)(\sigma^{v(F)}(t))\right )^{-1}} K^{\alpha,\beta,\gamma}(F,\sigma^{v(F)}((x-1)t)) +\left (K^{\alpha,\beta,\gamma}(F,\sigma^{v(F)}(xt))-K^{\alpha,\beta,\gamma}(F,\sigma^{v(F)}((x-1)t))\right)\right.\\
    &\quad \left.\left (I_{d-m}+D^2_{v(F)}(\Co^X)(\sigma^{v(F)}(t))\right )^{-1} \right)\times \left (I_{d-m}-D^2_{v(F)}(\Co^X)(\sigma^{v(F)}(t))\right)^{-1}.
  \end{align*}
  Similarly, we obtain
  \begin{align*}
    K^{\alpha,\beta,\gamma}&(F,\sigma^{v(F)}(xt))B(F,t) +K^{\alpha,\beta,\gamma}(F,\sigma^{v(F)}((x-1)t))A(F,t) \\
    &=\left(\vphantom{\left (I_{d-m}-D^2_{v(F)}\Co^X(\sigma^{v(F)}(t))\right)^{-1}} K^{\alpha,\beta,\gamma}(F,\sigma^{v(F)}(xt)) -\left (K^{\alpha,\beta,\gamma}(F,\sigma^{v(F)}(xt))-K^{\alpha,\beta,\gamma}(F,\sigma^{v(F)}((x-1)t))\right)\right.\\
     &\quad\left. \left (I_{d-m}+D^2_{v(F)}(\Co^X)(\sigma^{v(F)}(t))\right )^{-1} \right) \times\left (I_{d-m}-D^2_{v(F)}(\Co^X)(\sigma^{v(F)}(t))\right)^{-1}.
  \end{align*}
  We now use Lemma \ref{lem:matriceBounds} to bound
  \begin{align*}
    \| \left (I_{d-m}-D^2_{v(F)}\Co^X(\sigma^{v(F)}(t))\right)^{-1}\|
  \end{align*}
  and
  \begin{align*}
    \sup_{x\in[0,1]}& \|\left (K^{\alpha,\beta,\gamma}(F,\sigma^{v(F)}(xt))-K^{\alpha,\beta,\gamma}(F,\sigma^{v(F)}((x-1)t))\right)\left (I_{d-m}+D^2_{v(F)}\Co^X(\sigma^{v(F)}(t))\right )^{-1} \|
  \end{align*}
  for $t\in \overline{B^{d-m}_{2N}}$, independently of $F$. Whereas for the term $\|K^{\alpha,\beta,\gamma}(F,\sigma^{v(F)}(xt))\|$ and the term $\|K^{\alpha,\beta,\gamma}(F,\sigma^{v(F)}((x-1)t))\|$, we bound the directional derivatives by the partial ones and use the continuity with the estimates $\|\sigma^{v(F)}((x-1)t)\|\leq 2N$ and $\|\sigma^{v(F)}(xt)\|\leq 2N$, to bound their norms for $x\in [0,1]$, $t\in \overline{B^{d-m}_{2N}}$, independently of $F$.
\end{proof}

\section*{Acknowledgements}
  I am thankful to my advisors Prof. Dr. Daniel Hug and Prof. Dr. G\"unter Last for valuable input and discussions throughout this work. This work was supported by the German Research Foundation (DFG) [grant number FOR1548] through the research unit “Geometry and Physics of Spatial Random Systems”.

\providecommand{\bysame}{\leavevmode\hbox to3em{\hrulefill}\thinspace}
\providecommand{\MR}{\relax\ifhmode\unskip\space\fi MR }
\providecommand{\MRhref}[2]{%
  \href{http://www.ams.org/mathscinet-getitem?mr=#1}{#2}
}
\providecommand{\href}[2]{#2}

\bigskip
  \footnotesize

  \textsc{Karlsruhe Institute of Technology (KIT), Department of Mathematics, D-76131 Karlsruhe, Germany}\par\nopagebreak
  \textit{E-mail address:} \texttt{dennis.mueller@kit.edu}

\end{document}